\documentclass[
    pre,
    twocolumn,
    superscriptaddress,
    showpacs,
    nofootinbib,
    amsmath,
    amssymb,
    floatfix
]{revtex4-2}

\usepackage{graphicx}
\usepackage{hyperref}
\usepackage{xcolor}
\usepackage{bm}
\usepackage{mathtools}
\usepackage{booktabs}
\usepackage{algorithm}
\usepackage{algpseudocode}
\usepackage{amsmath}

\hypersetup{
    colorlinks=true,
    linkcolor=blue,
    citecolor=blue,
    urlcolor=blue
}

\newcommand{\R}{\mathbb{R}}

\newcommand{\K}{\mathcal{K}}
\newcommand{\F}{\mathbf{F}}        
\newcommand{\Faug}{\F^{\mathrm{aug}}}

\newcommand{\half}{\tfrac{1}{2}}
\newcommand{\dt}{\Delta t}

\usepackage{amsthm}

\theoremstyle{plain} 
\newtheorem{theorem}{\indent Theorem} 
\newtheorem{lemma}[theorem]{\indent Lemma}
\newtheorem{corollary}[theorem]{\indent Corollary}
\newtheorem{proposition}[theorem]{\indent Proposition}

\theoremstyle{definition} 
\newtheorem{definition}{\indent Definition}
\newtheorem{remark}{\indent Remark}

\begin{document}

\title{Forman--Ricci Curvature on Contact-Sequence Temporal Networks
       via Spatiotemporal Prism Complexes}

\author{Taiki Yamada}
\email[Corresponding author: ]{taiki\_yamada@riko.shimane-u.ac.jp}
\affiliation{Interdisciplinary Faculty of Science and Engineering, Shimane University,
Nishikawatsu 1060, Matsue, 690-8504, Japan}


\begin{abstract}
Temporal networks---sequences of time-stamped contacts among nodes---constitute the finest-grained representation of dynamic interaction data; however, geometric and topological analyses of such networks have remained largely confined to time-aggregated or snapshot-based approximations. 
Such reductions destroy the temporal ordering and interevent statistics essential for understanding spreading dynamics, synchronization, and information flow.  
This study proposes a geometric framework that lifts a contact-sequence temporal network into a genuine simplicial complex through a prism construction adapted from algebraic topology.  On this spatiotemporal prism complex, we develop the Forman--Ricci curvature in its original CW-complex form and contrast it with an augmented variant widely used in network science. 
We prove that the two variants coincide under uniform weights, derive a closed-form expression for their pointwise discrepancy in the general case, and identify the precise conditions under which they diverge---conditions generically satisfied in temporal networks because temporal edges carry interval-dependent weights.
Numerical experiments on three synthetic contact-network models (Erd\H os--R\'enyi, activity-driven, and bursty) and on the SocioPatterns Hypertext 2009 face-to-face contact dataset quantitatively confirm the theoretical predictions: the two Forman variants disagree on $56$--$67\%$ of the $1$-simplices---predominantly the temporal and diagonal simplices---while remaining strongly correlated according to the Pearson coefficient.
The proposed framework provides a principled, parameter-free method for assigning discrete Ricci curvature to each contact event, thereby opening a new geometric avenue for temporal data analysis.
\end{abstract}

\maketitle

\section{Introduction}\label{sec:intro}

Most empirical network data are intrinsically temporal \cite{Holme2012,Masuda2016}. Records of human face-to-face contacts, neuronal spike trains, financial transactions, online communications, and contact-tracing data naturally occur as \emph{contact sequences}: ordered triples $(i,j,t_k)$, specifying that nodes $i$ and $j$ interacted at time $t_k$.  This representation constitutes the most refined description available. Any coarser representation, including the familiar snapshot sequence and time-aggregated graph, can be recovered from the contact sequence through a (necessarily lossy) projection.

Despite its conceptual simplicity, the analytical framework for contact-sequence temporal networks remains comparatively underdeveloped. Most existing geometric and topological tools have been designed for static graphs and are therefore applied either to aggregated graphs or, at best, to sequences of static snapshots \cite{Krings2012,Ribeiro2013,Scholtes2014}. Accordingly, researchers have conflated two qualitatively distinct sources of information about a temporal network: 
\begin{enumerate} 
    \item the spatial structure of interactions at a given moment, and 
    \item the temporal ordering and interevent statistics of contacts, which govern causal reachability and burstiness \cite{Karsai2011,Lambiotte2013}.
\end{enumerate} 
Snapshot representations preserve (1) at a chosen temporal resolution but obliterate (2) within each observation window, whereas aggregated representations obliterate (2) entirely. Consequently, geometric quantities such as discrete curvature, when computed on snapshots, depend uncontrollably on temporal binning and cannot meaningfully reflect the causally ordered temporal structure.

In particular, discrete Ricci curvature has emerged as a powerful
descriptor of network geometry because it provides local indicators of
bottleneck structure, community separation, and robustness
\cite{Ollivier2009,Sandhu2016,Sreejith2016,Samal2018}.  Two principal
variants are currently in use: the optimal-transport-based Ollivier--Ricci
curvature \cite{Ollivier2009} and the combinatorial Forman--Ricci
curvature \cite{Forman2003}, with the latter originally formulated on
weighted CW complexes. In network science literature, an
augmented variant of Forman's curvature
\cite{Sreejith2016,Samal2018} is considerably more popular than the original
formula because it admits a particularly simple closed-form expression involving
only node degrees and triangle counts.  Whether this convenient
simplification preserves the geometric content of Forman's original
definition, and whether the answer changes when the underlying
complex is intrinsically temporal, has not, to the best of our knowledge, been
systematically investigated.

This study aims to overcome existing limitations.  Specifically, we
propose a mathematical framework in which contact-sequence temporal
networks are represented as genuine simplicial complexes through the
\emph{prism construction} of algebraic topology \cite{Hatcher2002}, thereby
yielding the \emph{spatiotemporal prism complex}.
$\K_{\mathrm{ST}}$.  On this complex, we compute discrete curvatures
and analyze how these values reflect both spatial and temporal
structure.  The main contributions of this study are as follows.
\begin{itemize}
\item[(i)] A definition of $\K_{\mathrm{ST}}$ that is rigorous,
      coordinate-free, and naturally inherits the time ordering of the
      underlying contact sequence.  In particular, $\K_{\mathrm{ST}}$
      is a bona fide simplicial complex (not merely a clique-rule
      shorthand) because its construction through prism subdivision
      automatically generates all faces of every dimension, including
      the diagonal $1$-simplices that mediate spatial connectivity
      and temporal persistence in the data.
\item[(ii)] A precise formulation of two Forman-type curvatures on
      $\K_{\mathrm{ST}}$---the original CW formula
      \cite{Forman2003} and the augmented form \cite{Sreejith2016}---along with the proof that they coincide under uniform weights
      (Lemma~\ref{lem:agreement}) and a closed-form expression for
      their pointwise discrepancy in the general case
      (Proposition~\ref{prop:discrepancy}).  As temporal edges
      carry weights that depend on the inter-contact interval, the
      two variants are generically inequivalent in temporal networks.
      We make this dependence explicit: spatial edges always satisfy
      $\F=\Faug$ on $\K_{\mathrm{ST}}$, whereas temporal and diagonal
      edges differ by an amount that is linear in $1-g(\dt)$ to leading
      order, where $g$ denotes the temporal-edge weight function and
      $\dt$ denotes the time gap.
\item[(ii$'$)] Two structural results exploit the prism subdivision itself beyond the $\F$ versus $\Faug$ comparison: a monotonicity statement asserting that persistent spatial simplices acquire non-negative curvature shifts from prism cofaces (Proposition~\ref{prop:monotonicity}), thereby formalizing why long-lived simplices appear less negatively curved than transient simplices in our experiments;
      and a discrete Gauss--Bonnet identity expressing the total Forman curvature of $\K_{\mathrm{ST}}$ as a sum of snapshot Euler characteristics corrected by persistent-simplex overlaps (Proposition~\ref{prop:gauss_bonnet}).
      The proof of Proposition~\ref{prop:monotonicity} relies on a curvature-decomposition lemma (Lemma~\ref{lem:coupling}) that decomposes the Forman--Ricci curvature of any spatial edge into a static intra-snapshot component and a controlled prism correction; this lemma is of independent interest because it isolates the geometric content of prism construction.
\item[(iii)] An empirical study of three synthetic contact-network
      models and the SocioPatterns Hypertext 2009 face-to-face
      contact dataset.  We find that the two Forman variants differ
      numerically on $56$--$67\%$ of $1$-simplices---almost entirely
      on temporal and diagonal edges, as the theory predicts---while remaining strongly correlated according to the Pearson coefficient
      ($\rho\geq 0.83$ in all four networks).  This result indicates that
      the augmented form can serve as a fast proxy for rank-based
      analyses, but the original form is required whenever curvature magnitudes are compared across
      edges with different temporal extents.
\end{itemize}

The remainder of this paper is organized as follows.
Section~\ref{sec:prelim} establishes the terminology, reviews
the background on simplicial complexes and Forman--Ricci curvature,
defines the spatiotemporal prism complex, and presents the two
curvature formulas that we compare.
Section~\ref{sec:theory_compare} establishes the equivalence and
inequivalence regimes between the two curvature variants, and proves
the supporting lemmas.
Section~\ref{sec:analysis} empirically analyzes the
synthetic and real-contact sequence datasets.
Section~\ref{sec:conclusion} discusses the methodological implications, positions our framework within existing temporal network geometry literature, and outlines future research directions.

\section{Preliminaries}\label{sec:prelim}
In this section, we present definitions of contact-sequence temporal networks, simplicial complexes, and Forman curvature. 

\subsection{Contact-sequence temporal networks}\label{ssec:contact}

\begin{definition}[Contact sequence]
A \emph{contact-sequence temporal network} on a node set $V$ is represented by the following set of triples.
\begin{equation}
    \mathcal{C} = \bigl\{(i_k,j_k,t_k) : k = 1,\dots,N\bigr\}
    \subset V \times V \times \R,
\end{equation}
where $i_k \neq j_k$, and the times $t_k$ are not assumed to be distinct. We denote by
\begin{equation}
    T(\mathcal{C}) = \{t : (i,j,t)\in\mathcal{C}\text{ for some }i,j\}
\end{equation}
the set of \emph{active times}, and by $T_v(\mathcal{C}) = \{t : (i,j,t)\in\mathcal{C}\text{ with }i=v\text{ or }j=v\}$ the set of times at which a node $v\in V$ is active.
\end{definition}

A contact sequence carries strictly more information than its snapshot or aggregate counterparts: any binning of the time axis yields a sequence of static graphs, but the converse mapping is lossy because intra-bin orderings are lost.  In particular, the notion of a \emph{time-respecting path}---a sequence
$(i_0,i_1,t_1),(i_1,i_2,t_2),\dots$ with $t_1<t_2<\cdots$---is intrinsic to $\mathcal{C}$ but is not, in general, recoverable from snapshot data \cite{Holme2015}.

\subsection{Simplicial complexes}\label{ssec:simpcx}

A finite \emph{abstract simplicial complex} on a vertex set $V_0$ is a collection $\K$ of nonempty finite subsets of $V_0$, called \emph{simplices}, that remains closed under inclusion.
For $\sigma\in\K$ with $|\sigma|=k+1$, $\sigma$ is a
\emph{$k$-simplex} written as $\dim\sigma = k$.  A simplex
$\tau\subsetneq\sigma$ is a \emph{face} of $\sigma$ and is written as
$\tau<\sigma$; conversely, $\sigma$ is a \emph{coface} of $\tau$, written as $\sigma>\tau$. 
When $\dim\sigma=\dim\tau+1$, $\tau$ is a codimension-$1$ face of $\sigma$ and $\sigma$ is a codimension-$1$ coface of $\tau$. 
Let $\K^{(k)}$ denote a set of $k$-faces of $\K$.
For $e \in \K^{(1)}$, we define $\mathrm{cof}_2(e)=\{T\in\K^{(2)}:T\supset e\}$.

A weighted simplicial complex assigns a positive real value $w(\sigma)>0$ to every simplex $\sigma\in\K$. 
Throughout this paper, all weights are assumed to be positive. When weights are not explicitly specified, we set $w(\sigma)\equiv 1$.

\begin{definition}[Flag complex]
The \emph{flag complex} of an undirected graph $G$ is a simplicial complex whose simplices are precisely the vertex sets of the cliques of $G$.  We denote this complex by $\mathrm{Fl}(G)$.
\end{definition}

\subsection{Spatiotemporal prism complexes}\label{ssec:prism}

Given a contact sequence $\mathcal{C}$, we construct a simplicial complex $\K_{\mathrm{ST}}(\mathcal{C})$ that encodes both the spatial structure of contemporaneous contacts and the persistence of those contacts across time. 
The construction follows Hatcher's prism operator from algebraic topology \cite[\S\,2.1]{Hatcher2002}.

For each $t\in T(\mathcal{C})$, the \emph{instantaneous contact graph} $G_t = (V_t, E_t)$ is defined as follows:
\begin{equation*}
    E_t = \bigl\{\{i,j\} : (i,j,t)\in\mathcal{C}\bigr\},
\end{equation*}
where $V_t$ denotes the set of endpoints appearing in $E_t$.  

We embed each node of $G_t$ as a \emph{spacetime vertex} $(v,t)\in V\times\R$. 
\begin{definition}[Spatial flag complex at each active time]
For each $t\in T(\mathcal{C})$, the \emph{spatial flag complex at time $t$} is defined as
\begin{equation*}
    F_t \;=\; \mathrm{Fl}(G_t) ,
\end{equation*}
with vertices $(v,t)$ for $v\in V_t$.
The simplices of $F_t$ are called \emph{spatial simplices}.
\end{definition}

Further, we introduce Hatcher's prism operator.
\begin{definition}[Hatcher's prism operator]\label{def:prism_op}
For each consecutive pair of active times $t<t'$ in $T(\mathcal{C})$, and for each simplex $\sigma = \left\{ v_0, v_1, \cdots, v_n \right\} \in F_t \cap F_{t'}$, \emph{Hatcher's prism operator} $\mathrm{Pr}(\sigma; t, t')$ triangulates $\sigma\times[t,t']$ into $(n+1)$ top-dimensional $(n+1)$-simplices
\begin{equation}\label{eq:prism}
    S_i = \{\hat v_0,\dots,\hat v_i,w_i,w_{i+1},\dots,w_n\},
    \quad i=0,\dots,n,
\end{equation}
where $w_i \mathrel{\mathop:}= (v_i,t')$ denotes the copy of $v_i$ at time $t'$, and $\hat v_i \mathrel{\mathop:}= (v_i,t)$.
The output $\mathrm{Pr}(\sigma; t, t')$ is a simplicial complex consisting of all faces of every simplix $S_i$.
\end{definition}

\begin{figure}[H]
\centering
\includegraphics[width=\columnwidth]{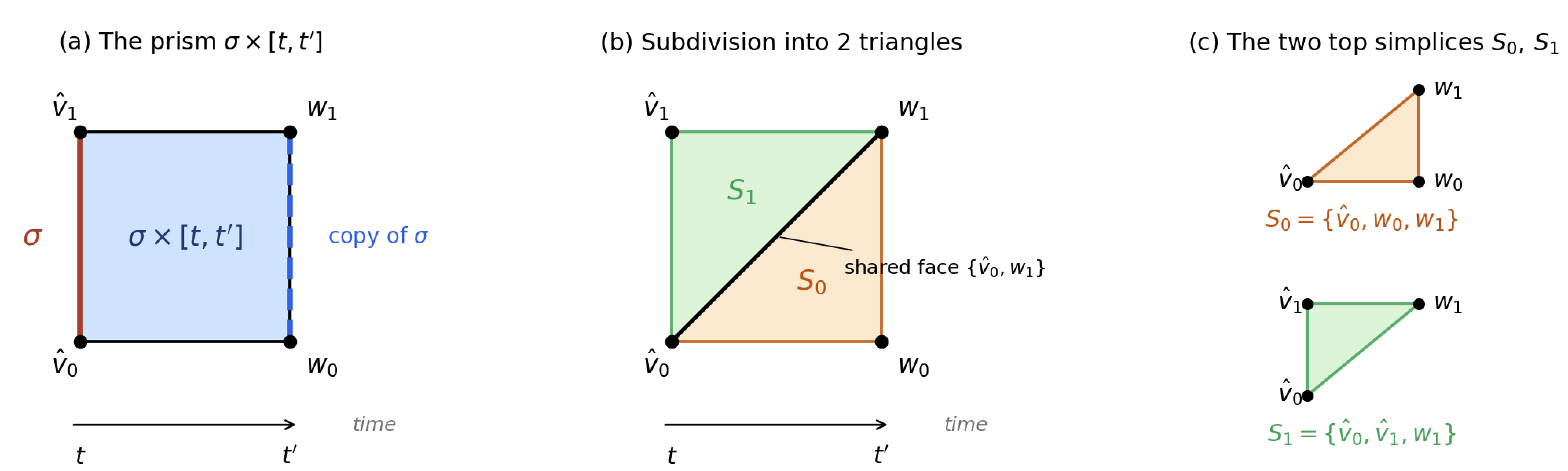}
\caption{Hatcher's prism operator for $n=1$}
\label{fig:prism operator}
\end{figure}

\begin{definition}[Spatiotemporal prism complex]\label{def:prism_cx}
Fix a positive integer $K$, representing the \emph{maximum slice gap}.
The \emph{spatiotemporal prism complex} associated with the contact sequence $\mathcal{C}$ and slice gap $K$ is a simplicial complex $\K_{\mathrm{ST}}(\mathcal{C})$, whose simplices are given by
\begin{equation}
    \K_{\mathrm{ST}}(\mathcal{C})
    \;=\;
    \bigcup_{t\in T(\mathcal{C})} F_t
    \;\cup\;
    \bigcup_{(t,t')\in\mathcal{T}_K}\bigcup_{\sigma\in F_t\cap F_{t'}}
        \mathrm{Pr}\bigl(\sigma;\,t,\,t'\bigr),
\end{equation}
where $\mathcal{T}_K = \{(t_i, t_j) \in T(\mathcal{C})^2 : 0 < j - i \le K\}$ is the set of active-time pairs whose slice indices differ by at most $K$, with the elements of $T(\mathcal{C})$ enumerated in an increasing order as $t_0 < t_1 < \cdots$.
The default choice $K = 1$ recovers the consecutive-pair construction,
whereas larger values of $K$ enrich $\K_{\mathrm{ST}}(\mathcal{C})$ with temporal and diagonal edges spanninglarger time gaps.
\end{definition}

\begin{remark}\label{rem:simplicial}
By construction, $\K_{\mathrm{ST}}(\mathcal{C})$ is closed under
taking faces and is therefore a simplicial complex.  The
$0$-simplices are spacetime vertices $(v,t)$, and the $1$-simplices fall
naturally into the following three classes:
\begin{description}
    \item[\textnormal{(spatial edges)}] $\{(i,t),(j,t)\}$ for some contact $(i,j,t)\in\mathcal{C}$,
    \item[\textnormal{(temporal edges)}] $\{(v,t),(v,t')\}$, produced by prism subdivision when $v$ persists from $t$ to $t'$,
    \item[\textnormal{(diagonal edges)}] $\{(i,t),(j,t')\}$ with $i\neq j$ and $t\neq t'$, generated by 
    Hatcher's prism operator.
\end{description}
The diagonal edges constitute the geometric feature that distinguishes the prism complex from a hybrid clique-rule construction, and they encode the simultaneity of spatial connectivity and temporal persistence.
\end{remark}

For this study, we assign default weights as follows.
\begin{itemize}
    \item For every spatial simplex $\sigma$, we set $w(\sigma)=1$.
    \item For a temporal edge $e=\{(v,t),(v,t')\}$, we set $w(e)=g(t'-t)$, where $g\colon\R_{>0}\to\R_{>0}$ is a fixed monotone-decreasing weight function (such as $g(\tau)=1/(1+\tau)$ or $g(\tau)=e^{-\lambda\tau}$). 
    \item For higher-dimensional simplices generated by prism subdivision, we use the geometric mean of their constituent edge weights.
\end{itemize}

\subsection{Forman--Ricci curvature: original and augmented forms}\label{ssec:forman}

This section introduces the two definitions of discrete Ricci curvature compared in this study. 
Both definitions apply to weighted CW complexes, of which simplicial complexes are a special class of topological spaces.

\begin{definition}[Parallel cells \cite{Forman2003}]\label{def:parallel}
Two $p$-cells $\alpha$ and $\hat\alpha$ in $\K$ are
\emph{parallel}, written $\alpha\Vert\hat\alpha$, if
$\hat\alpha\neq\alpha$ and exactly one of the following conditions
holds:
\begin{itemize}
    \item[\textnormal{(A)}] there exists a $(p+1)$-cell $\beta$ such that
          $\beta>\alpha$ and $\beta>\hat\alpha$;
    \item[\textnormal{(B)}] there exists a $(p-1)$-cell $\gamma$
          such that $\gamma<\alpha$ and $\gamma<\hat\alpha$.
\end{itemize}
\end{definition}
\begin{remark}\label{rem:parallel_simplicial}
On an abstract simplicial complex, condition~(A) for two $1$-cells implies condition~(B). Therefore, the XOR condition in Definition~\ref{def:parallel}, when restricted to $1$-cells, reduces to the requirement that the cells ``share a common vertex but not a common triangle.''
\end{remark}

\begin{definition}[Forman--Ricci curvature, original CW form
\cite{Forman2003}]\label{def:forman_orig}
Let $\K$ be a weighted CW complex.
For a $p$-cell $\alpha\in\K$, the \emph{Forman--Ricci curvature} of $\alpha$ is defined as:
\begin{multline}\label{eq:forman_orig}
\F(\alpha) = w(\alpha)\Biggl[
\sum_{\beta>\alpha}\frac{w(\alpha)}{w(\beta)}
+ \sum_{\gamma<\alpha}\frac{w(\gamma)}{w(\alpha)} \\
{}- \sum_{\hat\alpha\,\Vert\,\alpha}\,
\biggl|
\sum_{\beta>\alpha,\hat\alpha}\!
       \frac{\sqrt{w(\alpha)w(\hat\alpha)}}{w(\beta)}
- \sum_{\gamma<\alpha,\hat\alpha}\!
       \frac{w(\gamma)}{\sqrt{w(\alpha)w(\hat\alpha)}}
\biggr|
\Biggr].
\end{multline}
where $\beta>\alpha$ ranges over the codimension-$1$ cofaces of $\alpha$, and $\gamma<\alpha$ ranges over the codimension-$1$ faces of $\alpha$.
\end{definition}

The augmented form replaces the absolute-value term with a simpler expression and is the variant most frequently used in network-science literature.
\begin{definition}[Augmented Forman--Ricci curvature
\cite{Sreejith2016,Samal2018}]\label{def:forman_aug}
For a $1$-simplex $e = \left\{ v_1, v_2 \right\}\in\K^{(1)}$,
\begin{eqnarray}\label{eq:forman_aug}
    \Faug(e) &=& w(e)\!\Biggl[
        \frac{w(v_1)}{w(e)} + \frac{w(v_2)}{w(e)} \nonumber\\
    &\ &    + \!\sum_{T>e}\!\frac{w(e)}{w(T)} - \!\sum_{e'\Vert' e}\!\sqrt{\frac{w(e)}{w(e')}}
    \Biggl],
\end{eqnarray}
where $T$ ranges over $2$-simplices containing $e$.
\end{definition}

\section{Comparison with the two Forman--Ricci curvature on the prism complex}\label{sec:theory_compare}
In this section, we compare the original and augmented Forman--Ricci curvatures, $\F$ and $\Faug$, on the spatiotemporal prism complex.
We begin with a preparatory observation concerning the parallelism relation in Definition~\ref{def:parallel}, followed by a proof of the agreement statement in Lemma~\ref{lem:agreement}.
\subsection{Agreement under uniform weights}\label{ssec:agreement}

The XOR condition in Definition~\ref{def:parallel} is stated in the general setting of CW complexes, and on a generic CW complex, both clauses (A) and (B) carry independent geometric content.
On an abstract simplicial complex---such as $\K_{\mathrm{ST}}$---the two clauses are not independent: condition (A) implies condition (B).

\begin{lemma}\label{lem:reduction}
Let $\K$ be an abstract simplicial complex, and let $\alpha,\hat\alpha$ be two distinct $1$-simplices in $\K$.
Then the following equivalence holds:
\begin{equation*}
    \hat\alpha\Vert\alpha
    \iff
    \alpha\cap\hat\alpha\neq\emptyset
    \;\text{and}\;
    \nexists\, \beta\in\K^{(2)}\text{ with }\alpha,\hat\alpha\subset\beta.
\end{equation*}
\end{lemma}
\begin{proof}
Assuming that condition (A) holds, there exists a $2$-simplex $\beta\in\K$ containing both $\alpha$ and $\hat\alpha$.
Then $\beta$ has exactly three vertices, and the two distinct $1$-simplices $\alpha,\hat\alpha\subset\beta$ must share at least one vertex.
This common vertex is a $0$-simplex contained in both $\alpha$ and $\hat\alpha$, thereby establishing condition (B).
Thus, the XOR condition in Definition~\ref{def:parallel} reduces to the statement that``condition (B) holds and condition (A) fails.''
\end{proof}

We abbreviate the parallel relation of simplicial complexes by defining
\begin{equation}
    \mathcal{P}(\alpha)
    = \bigl\{\hat\alpha\in\K^{(1)} :
            \hat\alpha\Vert\alpha\bigr\}.
\end{equation}

Therefore, we can prove Lemma~\ref{lem:agreement}.

\begin{lemma}\label{lem:agreement}
Let $\K$ be a finite simplicial complex on which all simplex weights equal the same positive constant.
Subsequently, for every $1$-simplex $e\in\K^{(1)}$,
\begin{equation*}
    \F(e) \;=\; \Faug(e).
\end{equation*}
\end{lemma}

\begin{proof}
Let $w_0>0$ denote the common weight assigned to all simplices, and fix a $1$-simplex $e\in\K^{(1)}$ with vertices $v_1,v_2$.
We have $\sqrt{w(\alpha)w(\hat\alpha)} = w_0$ for any pair of $1$-simplices $\alpha,\hat\alpha$.

The first two terms in Equation ~\eqref{eq:forman_orig} evaluate to
\begin{equation}\label{eq:trivial_uniform}
    \sum_{T>e}\!\frac{w_0}{w_0}+\sum_{v<e}\!\frac{w_0}{w_0}
    = |\mathrm{cof}_2(e)|+2.
\end{equation}
The corresponding brackets in Equation~\eqref{eq:forman_aug} include
\begin{equation}
    \frac{w_0}{w_0}+\frac{w_0}{w_0}+\sum_{T>e}\!\frac{w_0}{w_0}
    = 2+|\mathrm{cof}_2(e)|,
\end{equation}
which is identical to Equation~\eqref{eq:trivial_uniform}.

For the third and fourth terms in Equation~\eqref{eq:forman_orig}, Lemma~\ref{lem:reduction} implies that every parallel $1$-simplex $\hat e\in\mathcal P(e)$ shares exactly one vertex with $e$ and shares no $2$-simplex with $e$.
Hence, in Equation ~\eqref{eq:forman_orig}, the third term $\sum_{\beta>e,\hat e}\!\sqrt{w(e)w(\hat e)}/w(\beta)$ vanishes, whereas the fourth term reduces to a single contribution:
\begin{equation*}
\sum_{\substack{\gamma<e,\hat e}}\!\frac{w(\gamma)}{\sqrt{w(e)w(\hat e)}}
= \frac{w_0}{w_0} = 1.
\end{equation*}
The corresponding sum in Equation~\eqref{eq:forman_aug} is
\begin{equation*}
\sum_{\hat e\in\mathcal P(e)}\!\sqrt{w_0/w_0}
= |\mathcal P(e)|.
\end{equation*}

Combining these results, we obtain
\begin{equation}\label{eq:uniform_value}
\F(e) \,=\, \Faug(e) \,=\, w_0\bigl(2+|\mathrm{cof}_2(e)|-|\mathcal P(e)|\bigr).
\end{equation}
\end{proof}
\begin{remark}
The converse statement is false: the two curvatures generically disagree as soon as either the vertex weights or the edge weights cease to remain constant.
On the spatiotemporal prism complex $\K_{\mathrm{ST}}(\mathcal{C})$, this behavior is the rule rather than the exception because temporal edges naturally inherit a positive interval-dependent weight $g(t'-t)$.

In the unit-weight case $w_0=1$, Equation~\eqref{eq:uniform_value} reduces to $\F(e)=2+\#\text{triangles}_e-\#\text{parallel}_e$, which is the familiar combinatorial Forman expression used as a definition in network-science literature \cite{Sreejith2016}.
Proof shows that this formula precisely characterizes the regime in which the original and augmented variants agree.
\end{remark}

\subsection{Closed form for the discrepancy and its time-gap dependence}\label{ssec:discrepancy}

We computed the edge-by-edge discrepancies in the general weighted setting.
Let $e\in\K^{(1)}$ have vertices $v_1 and v_2$ with weights $w(v_1)$ and $w(v_2)$, respectively, and edge weight $w(e)$. Let $\{T_1,\ldots,T_p\}=\mathrm{cof}_2(e)$ denote the collection of $2$-simplices.
From Lemma~\ref{lem:reduction}, every parallel $\hat e\in\mathcal P(e)$ shares exactly one vertex with $e$.
We denote $v(\hat e)$ as the common vertex and $w(\hat e)$ as the weight of the parallel edge.

\begin{proposition}\label{prop:discrepancy}
For any $e\in\K^{(1)}$ in a weighted simplicial complex,
\begin{equation*}
\F(e)-\Faug(e) =
\sum_{\hat e\in\mathcal P(e)}
   \!\Biggl[
     \sqrt{\frac{w(e)\,w(v(\hat e))^2}{w(\hat e)}}
     \;-\;
     w(e)\sqrt{\frac{w(e)}{w(\hat e)}}
   \Biggr].
\end{equation*}
\end{proposition}
The proof follows immediately from Lemma~\ref{lem:reduction}, as Equation~\eqref{eq:forman_orig} reduces to
\begin{multline*}
\F(e) = w(e) \biggr\{\sum_{T>e}\cfrac{w(e)}{w(T)} + \cfrac{w(v_1)}{w(e)} + \cfrac{w(v_2)}{w(e)}\\
+ \sum_{\hat e\in\mathcal P(e)}\sqrt{\frac{w(e)\,w(v(\hat e))^2}{w(\hat e)}} \biggr\}.
\end{multline*}
Subtracting Equation~\eqref{eq:forman_aug} from this expression yields the desired result.

\begin{remark}
The summand in Proposition~\ref{prop:discrepancy} factors as
\begin{equation}\label{eq:summand_factored}
\sqrt{\frac{w(e)}{w(\hat e)}}
\Bigl[w(v(\hat e)) - w(e)\Bigr],
\end{equation}
This factorization makes the algebraic origin of the discrepancy transparent: each parallel edge contributes a term proportional to the difference between the common vertex weight and the edge weight $w(e)$, modulated by the geometric ratio $\sqrt{w(e)/w(\hat e)}$.
\end{remark}

This factored form yields three immediate corollaries.
\begin{corollary}\label{cor:no_parallel}
If $\mathcal P(e)=\emptyset$, then $\F(e)=\Faug(e)$.
\end{corollary}
\begin{corollary}\label{cor:bound}
The two curvatures satisfy the inequality.
\begin{equation}\label{eq:abs_bound}
|\F(e)-\Faug(e)|
\,\leq\,
\sum_{\hat e\in\mathcal P(e)}
\sqrt{\frac{w(e)}{w(\hat e)}}\,
\bigl|w(v(\hat e))-w(e)\bigr|.
\end{equation}
\end{corollary}

Further, we specialize Proposition~\ref{prop:discrepancy} to $\K_{\mathrm{ST}}(\mathcal{C})$ using the default weights introduced in Section ~\ref{ssec:prism}. Vertices satisfy $w(v)=1$, spatial edges satisfy $w(e)=1$, temporal edges $\{(v,t),(v,t')\}$ satisfy $w(e)=g(t'-t)$, where $g\colon\R_{>0}\to\R_{>0}$ denotes monotone decreasing, and diagonal edges $\{(i,t),(j,t')\}$ satisfy $w(e)=\half g(t'-t)$.

\begin{corollary}\label{cor:prism_specialization}
Let $e \in \K_{\mathrm{ST}}^{(1)}$. Assuming the default weights from Section ~\ref{ssec:prism}, Proposition~\ref{prop:discrepancy} specializes as follows:
\begin{enumerate}
\item[\textnormal{(i)}] If $e$ is a spatial edge, then $\F(e) = \Faug(e)$.
\item[\textnormal{(ii)}] If $e$ is a temporal or diagonal edge with time gap $\Delta t > 0$, then
\begin{equation}\label{eq:temp_diag_disc}
\F(e) - \Faug(e) = \sum_{\hat e \in \mathcal{P}(e)} \sqrt{\frac{w(e)}{w(\hat e)}}\bigl(1 - w(e)\bigr).
\end{equation}
In particular, the discrepancy is non-zero when $g(\Delta t) \neq 1$.
\end{enumerate}
\end{corollary}

\begin{proof}
We substitute $w(v(\hat e)) = 1$ into Proposition~\ref{prop:discrepancy} for all $\hat e$.
The summand factor $w(v(\hat e)) - w(e)$ becomes $1 - w(e)$.
For (i), $w(e) = 1$ implies $1 - w(e) = 0$, and thus each summand vanishes.
For (ii), $w(e) \in \{g(\Delta t), \frac{1}{2}g(\Delta t)\} \neq 1$ for $\Delta t > 0$ (provided $g(\Delta t) \neq 1$), yielding Equation ~\eqref{eq:temp_diag_disc}.
\end{proof}

The factor $(1-w(e))$ in Equation~\eqref{eq:temp_diag_disc} captures the underlying mechanism: as the time gap $\dt$ increases, $w(e)=g(\dt)$ decreases; consequently, $(1-w(e))$ increases, and the discrepancy grows proportionally.
When $\dt\to 0^+$, corresponding to instantaneous persistence, $w(e)\to 1$, and the discrepancy vanishes.

Corollary~\ref{cor:prism_specialization} provides a complete characterization of when and why the original and augmented Forman variants agree or disagree on the prism complex.
\begin{itemize}
\item Spatial edges contribute identical values under both curvature definitions, regardless of the choice of $g$.
\item Temporal and diagonal edges contribute identical values only when $g\equiv 1$, corresponding to the absence of time-gap dependence; otherwise, they disagree by an amount that is linear in $1-g(\dt)$ to leading order.
\end{itemize}
The empirical analysis presented in Section~\ref{sec:analysis} quantitatively confirms both predictions.

\subsection{Monotonicity for persistent simplices}\label{ssec:monotonicity}

A structural property of the prism complex emerges when we track the curvature of a single spatial simplex as it persists across several time slices.
Intuitively, the longer a simplex persists, the more $2$-cofaces it accumulates from prism subdivisions on adjacent time slices, and these additional cofaces shift its curvature in a non-negative direction under suitable assumptions.

Further, we fix the notation used throughout this subsection.
Let $T(\mathcal{C}) = \{t_0, t_1, \ldots, t_M\}$ be the set of active times listed in increasing order, where $t_k$ denotes the time corresponding to slice index $k$.
Let $e = \{u, v\}$ denote a pair of nodes that form an edge of $V$, and define
\begin{equation*}\label{eq:embedding}
e^{(k)} := \{(u, t_k),\,(v, t_k)\}
\end{equation*}
to embed a spatial edge in $\K_{\mathrm{ST}}$ at time $t_k$.
Given a $1$-simplex $e^{(k)}$, we define
\begin{align*}
\mathrm{cof}_2^{\mathrm{prism}}(e^{(k)}) &:= \{T \in \K_{\mathrm{ST}}^{(2)} : T \supset e^{(k)},\ T \notin F_{t_k}\},\\
\mathcal{P}_{\mathrm{prism}}(e^{(k)}) &:= \{\hat e \in \mathcal{P}(e^{(k)}) : \hat e \notin F_{t_k}\}, 
\end{align*}
as the sets of $2$-cofaces and parallel $1$-simplices of $e^{(k)}$ that lie outside the static flag complex $F_{t_k}$.
We refer to these sets as the prism-origin coface set and prism-origin parallel set, respectively.

We define $\F_{F_{t_k}}(e)$ for the Forman--Ricci curvature of $e = \{u, v\}$ computed solely within $F_{t_k}$, and $\F_{\K_{\mathrm{ST}}}(e^{(k)})$ for the Forman--Ricci curvature of $e^{(k)}$ computed within $\K_{\mathrm{ST}}$.

\begin{lemma}\label{lem:coupling}
Let $e^{(k)} \in \K_{\mathrm{ST}}^{(1)}$ be a spatial edge.
Then,
\begin{equation*}\label{eq:coupling}
\F_{\K_{\mathrm{ST}}}(e^{(k)}) \;=\; \F_{F_{t_k}}(e) \;+\; \Delta_{\mathrm{prism}}(e^{(k)}),
\end{equation*}
where the prism correction term is defined as:
\begin{eqnarray*}\label{eq:prism_correction}
\Delta_{\mathrm{prism}}(e^{(k)}) &=& w(e^{(k)})\!\sum_{T\in\mathrm{cof}_2^{\mathrm{prism}}(e^{(k)})}\!\frac{w(e^{(k)})}{w(T)} \\
&\ & - w(e^{(k)})\,\Omega(e^{(k)}),
\end{eqnarray*}
and
\begin{equation*}\label{eq:omega_def}
\Omega(e^{(k)}) \;:=\; \sum_{\hat e \in \mathcal{P}_{\mathrm{prism}}(e^{(k)})} \frac{w(v(\hat e))}{\sqrt{w(e^{(k)})\,w(\hat e)}}.
\end{equation*}
\end{lemma}
The conjunction “and” was replaced with “whereas” because the sentence contrasts two different behaviors of the sums: one disappears entirely while the other contributes a positive term. “Whereas” highlights this contrast more effectively.\begin{proof}
We apply Equation~\eqref{eq:forman_orig} to $e^{(k)}$ within $\K_{\mathrm{ST}}$.
By Lemma~\ref{lem:reduction}, every parallel edge $\hat e \in \mathcal{P}(e^{(k)})$ shares exactly one vertex $v(\hat e)$ with $e^{(k)}$ and shares no triangle with $e^{(k)}$.
Hence, in the absolute-value bracket of Equation~\eqref{eq:forman_orig}, the sum over common $2$-cofaces vanishes, whereas the sum over common $0$-faces reduces to the single positive term $w(v(\hat e))/\sqrt{w(e^{(k)})\,w(\hat e)}$.
Considering the absolute value preserves this positive contribution, and the parallel-edge correction term in Equation ~\eqref{eq:forman_orig} simplifies to:
\begin{equation}\label{eq:par_simplified}
\sum_{\hat e\in\mathcal{P}(e^{(k)})} \frac{w(v(\hat e))}{\sqrt{w(e^{(k)})\,w(\hat e)}}.
\end{equation}

Further, we decompose each contribution in Equations ~\eqref{eq:forman_orig} according to whether the contributing simplex lies in $F_{t_k}$.
The trivial $0$-face contribution is identical in $\F_{F_{t_k}}(e)$ and $\F_{\K_{\mathrm{ST}}}(e^{(k)})$, as $e^{(k)}$ has the same two endpoints in both complexes.
The trivial $2$-coface contribution decomposes as follows:
\begin{eqnarray*}
\sum_{T \in \mathrm{cof}_2^{\K_{\mathrm{ST}}}(e^{(k)})}\!\frac{w(e^{(k)})}{w(T)}
&=&
\sum_{T \in \mathrm{cof}_2^{F_{t_k}}(e)}\!\frac{w(e^{(k)})}{w(T)}\\
&\ &+ \sum_{T \in \mathrm{cof}_2^{\mathrm{prism}}(e^{(k)})}\!\frac{w(e^{(k)})}{w(T)},
\end{eqnarray*}
the first term is a component of $\F_{F_{t_k}}(e)/w(e^{(k)})$, while the second term is the upward component of $\Delta_{\mathrm{prism}}/w(e^{(k)})$.
The parallel-edge sum in Equation ~\eqref{eq:par_simplified} is decomposed analogously.
\footnotesize
\begin{eqnarray*}
\sum_{\hat e \in \mathcal{P}(e^{(k)})} \frac{w(v(\hat e))}{\sqrt{w(e^{(k)})\,w(\hat e)}}
&=&
\sum_{\hat e \in \mathcal{P}(e^{(k)}) \setminus \mathcal{P}_{\mathrm{prism}}(e^{(k)})} \frac{w(v(\hat e))}{\sqrt{w(e^{(k)})\,w(\hat e)}}\\
&\ &+ \Omega(e^{(k)}),
\end{eqnarray*}
\normalsize
where the first sum on the right consists of spatially parallel edges and contributes (with a negative sign, as in Equation ~\eqref{eq:forman_orig}) to $\F_{F_{t_k}}(e)/w(e^{(k)})$. The second component involves $\Omega(e^{(k)})$.
This completes the Proof.
\end{proof}

Further, we state and prove the results of monotonicity.

\begin{theorem}\label{prop:monotonicity}
Let $\K_{\mathrm{ST}}(\mathcal{C})$ be constructed using a slice gap $K \ge 1$ and uniform weights satisfying $w \equiv 1$.
Let $e^{(k)} = \{(u, t_k), (v, t_k)\}$ denote a spatial edge of $\K_{\mathrm{ST}}$ at time $t_k$.
Assuming that the following conditions hold:
\begin{enumerate}
\item[\textnormal{(A1)}] there exists $\delta \in \{1, \ldots, K\}$ such that $e \in F_{t_{k+s}}$ for every $|s| \le \delta$,
\item[\textnormal{(A2)}] every $\hat e \in \mathcal{P}_{\mathrm{prism}}(e^{(k)})$ is generated by Hatcher's prism operator applied to a persistent simplex that contains both endpoints of $e$.
\end{enumerate}
Then,
\begin{eqnarray}\label{eq:monotonicity_main}
\F_{\K_{\mathrm{ST}}}(e^{(k)}) - \F_{F_{t_k}}(e)
&=& |\mathrm{cof}_2^{\mathrm{prism}}(e^{(k)})| - |\mathcal{P}_{\mathrm{prism}}(e^{(k)})|\nonumber\\
&\ge& 0.
\end{eqnarray}
The equality holds if and only if every $T \in \mathrm{cof}_2^{\mathrm{prism}}(e^{(k)})$ contains a temporal edge as one of its two non-$e^{(k)}$ edges.
\end{theorem}

\begin{proof}
Under uniform weights, Lemma~\ref{lem:agreement} yields the closed-form expression:
$\F(e) = 2 + |\mathrm{cof}_2(e)| - |\mathcal{P}(e)|$ for any $1$-simplex $e$ in the simplicial complex.
Applying this expression to $e^{(k_0)} \in \K_{\mathrm{ST}}$ and $e \in F_{t_{k_0}}$ yields
\begin{align*}
\F_{\K_{\mathrm{ST}}}(e^{(k)}) &= 2 + |\mathrm{cof}_2^{\K_{\mathrm{ST}}}(e^{(k)})| - |\mathcal{P}(e^{(k)})|, \\
\F_{F_{t_k}}(e) &= 2 + |\mathrm{cof}_2^{F_{t_k}}(e)| - |\mathcal{P}_{F_{t_k}}(e)|.
\end{align*}
Therefore,
\begin{eqnarray}\label{eq:result_step1}
\F_{\K_{\mathrm{ST}}}(e^{(k)}) - \F_{F_{t_{k}}}(e) &=& (|\mathrm{cof}_2^{\K_{\mathrm{ST}}}(e^{(k)})| - |\mathrm{cof}_2^{F_{t_k}}(e)|)\nonumber\\
&-& (|\mathcal{P}(e^{(k)})| - |\mathcal{P}_{F_{t_k}}(e)|).
\end{eqnarray}
The cofaces of $e^{(k_0)}$ in $\K_{\mathrm{ST}}$ are categorized into spatial $T \in F_{t_{k_0}}$ and non-spatial $T \in \mathrm{cof}_2^{\mathrm{prism}}(e^{(k_0)})$. Then,
\begin{eqnarray*}
|\mathrm{cof}_2^{\K_{\mathrm{ST}}}(e^{(k)})| &=& \{T \in \K_{\mathrm{ST}}^{(2)} : T \supset e^{(k)},\ T \in F_{t_k}\}\\
&+& \{T \in \K_{\mathrm{ST}}^{(2)} : T \supset e^{(k)},\ T \notin F_{t_k}\}\\
&=& |\mathrm{cof}_2^{F_{t_k}}(e)| + |\mathrm{cof}_2^{\mathrm{prism}}(e^{(k)})|.
\end{eqnarray*}

For the parallel sets, we first observe  that any $T \in \mathrm{cof}_2^{\mathrm{prism}}(e^{(k)})$ has at least one vertex outside the time slice $t_k$ by definition.
By inspecting the top-dimesnional simplices in Equation~\eqref{eq:prism}, we see that any prism-origin 2-simplex contains at most one $1$-simplex whose endpoints both lie in the time slice $t_k$, namely the bottom face of the prism.
Therefore, $T$ cannot simultaneously contain $e^{(k)}$ and another spatial $1$-simplex $\hat e \in F_{t_k}$.
This argument shows that any $\hat e \in F_{t_k} \cap \mathcal{P}_{F_{t_k}}(e)$ remains parallel to $e^{(k)}$ in $\K_{\mathrm{ST}}$, since no new common 2-coface is created. Consequently, the spatial parallel set is in bijective correspondence with $\mathcal{P}_{F_{t_k}}(e)$.
Conversely, assumption (A2) implies that any non-spatial parallel edge $\hat e \in \mathcal{P}(e^{(k)})$ with $\hat e \notin F_{t_k}$ is generated by prism subdivision and therefore belongs to $\mathcal{P}_{\mathrm{prism}}(e^{(k)})$.
Combining these observations, we have
\begin{equation*}
|\mathcal{P}(e^{(k)})| - |\mathcal{P}_{F_{t_k}}(e)| = |\mathcal{P}_{\mathrm{prism}}(e^{(k)})|.
\end{equation*}
Therefore, we obtain Equation ~\eqref{eq:monotonicity_main}.

Further, we show $|\mathcal{P}_{\mathrm{prism}}(e^{(k)})| \le |\mathrm{cof}_2^{\mathrm{prism}}(e^{(k)})|$.
We construct an injection map as $\Phi: \mathcal{P}_{\mathrm{prism}}(e^{(k)}) \to \mathrm{cof}_2^{\mathrm{prism}}(e^{(k)})$.
Let $\hat e \in \mathcal{P}_{\mathrm{prism}}(e^{(k)})$.
By Lemma~\ref{lem:reduction}, $\hat e$ shares exactly one vertex with $e^{(k)}$. Without loss of generality, assuming that this common vertex is $(u, t_k)$, ssuch that $\hat e = \{(u, t_k), (w_*, t_{k+s_*})\}$ for some $w_* \in V$ satisfying $|s_*| \le \delta$.
By assumption (A2), there exists a persistent simplex $\sigma$ such that $\{u, w_*\} \subseteq \sigma$, $\sigma \in F_{t_k} \cap F_{t_{k+s_*}}$. This simplex generates $\hat e$ through prism subdivision. Assumption (A2) further implies that $v \in \sigma$; therefore, $\{u, v, w_*\} \in F_{t_k} \cap F_{t_{k+s_*}}$.
Further, we define
\begin{equation*}
\Phi(\hat e) := \{(u, t_k), (v, t_k), (w_*, t_{k+s_*})\} \in \K_{\mathrm{ST}}^{(2)},
\end{equation*}
whose three edges are $e^{(k)}$, $\hat e$, and $\{(v, t_k), (w_*, t_{k+s_*})\}$, where the last edge arises from prism subdivision applied to $\{u, v, w_*\}$.
Moreover, $\Phi(\hat e) \notin F_{t_k}$, as one of its vertices lies at time $t_{k+s_*} \neq t_k$. Therefore, $\Phi(\hat e) \in \mathrm{cof}_2^{\mathrm{prism}}(e^{(k)})$.

To demonstrate that $\Phi$ is injective, we assume that $\Phi(\hat e_1) = \Phi(\hat e_2) = T$.
The simplex $T$ has exactly three edges, two of which are $e^{(k)}$ and the corresponding preimage edge by construction, leaving only one remaining edge, denoted by $f$.
If $\hat e_1$ and $\hat e_2$ are distinct preimages, then one of them can coincide with $f$. However, $f$ shares a vertex with $e^{(k)}$, and both edges lie in the common $2$-coface $T$.
From Lemma~\ref{lem:reduction}, we have $f \notin \mathcal{P}(e^{(k)})$, which contradicts the assumption that $f = \hat e_i \in \mathcal{P}(e^{(k)})$.
Hence, $\hat e_1 = \hat e_2$.

Therefore, the injection map $\Phi$ yields
\begin{equation*}
|\mathcal{P}_{\mathrm{prism}}(e^{(k)})| \le |\mathrm{cof}_2^{\mathrm{prism}}(e^{(k)})|,
\end{equation*}
thereby completing Equation ~\eqref{eq:monotonicity_main}.

Furthermore, $\Phi$ is a bijection if and only if every $T \in \mathrm{cof}_2^{\mathrm{prism}}(e^{(k)})$ lies in the image.
By the preceding argument, among the two non-$e^{(k)}$ edges of $T$, at least one fails to be parallel to $e^{(k)}$ because it shares the common $2$-coface $T$.
Hence, the only way for the remaining edge to be parallel to $e^{(k)}$ and to lie in the image of $\Phi$ is for it to be a temporal edge $\{(v, t_k), (v, t_{k+s_*})\}$ (or the analogous edge associated with $u$) generated solely by the persistence of $\{v\}$.
This condition is precisely the equality condition stated in the theorem.
\end{proof}

\begin{remark}\label{rem:monotonicity_interpretation}
Theorem~\ref{prop:monotonicity} formalizes the heuristic principle that ``persistent simplices accumulate non-negative curvature contributions from prism cofaces.''
Assumption (A2) restricts attention to non-spatial parallel partners that genuinely arise from prism construction, which constitutes the dominant case in temporal networks with localized contact patterns.
This mechanism underlines the empirical observations shown in Figure ~\ref{fig:by_kind}. The bursty model and the HT09 dataset exhibit less negative temporal-edge means than the ER baseline because persistent edges in those networks accumulate more prism cofaces relative to their parallel partners.
\end{remark}

\subsection{A combinatorial Gauss--Bonnet identity}\label{ssec:gauss_bonnet}

Forman's original framework includes a discrete Gauss--Bonnet theorem, which states that for a finite cell complex without a boundary, the alternating sum of the curvatures of all $p$-cells equals the Euler characteristic of the complex.
In the spatiotemporal prism complex, the same identity admits a transparent decomposition into spatial and prism contributions.

\begin{theorem}\label{prop:gauss_bonnet}
Let $\K_{\mathrm{ST}}(\mathcal{C})$ be a finite spatiotemporal prism complex equipped with the uniform weight function $w \equiv 1$. The alternating Forman sum is defined as:
\begin{equation*}\label{eq:GB_lhs}
\mathfrak{F}(\K_{\mathrm{ST}}) \;:=\;
\sum_{p\geq 0}(-1)^p\!\sum_{\alpha\in\K_{\mathrm{ST}}^{(p)}}\!\F(\alpha).
\end{equation*}
Then, we have
\begin{eqnarray}\label{eq:GB_decomp}
\mathfrak{F}(\K_{\mathrm{ST}}) &=& \chi(\K_{\mathrm{ST}})\nonumber\\
&=& \sum_{t\in T(\mathcal{C})}\chi(F_t) \;-\; \sum_{(t,t')\in\mathcal{T}_K}\chi\bigl(F_t\cap F_{t'}\bigr)\nonumber\\
&\ & + \mathcal{R}(\mathcal{C}),
\end{eqnarray}
where $\mathcal{R}(\mathcal{C})$ collects the higher-order inclusion--exclusion remainder terms arising from intersections of three or more persistent simplices, and $\mathcal{T}_K$ denotes the set of slice pairs introduced in Definition~\ref{def:prism_cx}.
\end{theorem}

\begin{proof}[Proof sketch]
The first equality, $\mathfrak{F}(\K_{\mathrm{ST}}) = \chi(\K_{\mathrm{ST}})$
follows from the discrete Gaussian--Bonnet theorem of Forman~\cite{Forman2003} applied
to the simplicial complex $\K_{\mathrm{ST}}$.

To establish the second equality, we use two key observations.

(a) Each prism $\mathrm{Pr}(\sigma; t, t')$ is a cone with apex $\hat v_0$.
Consequently, $\chi(\mathrm{Pr}(\sigma; t, t')) = 1$.

(b) The intersection $\mathrm{Pr}(\sigma; t, t') \cap F_t$ equals the embedding
of the face complex of $\sigma$ into $F_t$, and therefore has Euler characteristic $\chi = 1$.

Applying the inclusion--exclusion principle to the cover
$\K_{\mathrm{ST}} = \bigcup_t F_t \cup \bigcup_{(t,t')} \mathrm{Pr}_{t,t'}$
and using observations (a) and (b) to compute the Euler characteristics of the pairwise
intersections, we find that the prism contributions cancel between the bottom and top intersections, leaving only the snapshot Euler characteristics.
The remaining terms are the snapshot Euler characteristics $\chi(F_t)$ along with the persistent-overlap corrections $\chi(F_t \cap F_{t'})$.
Higher-order terms, corresponding to subsets satisfying ($|A| \geq 3$ in the inclusion--exclusion expansion,
along with pairs $(t, t')$ outside $\mathcal{T}_K$), are collected into the
remainder term $\mathcal{R}(\mathcal{C})$.  

The detailed calculations are provided in Appendix~\ref{app:gauss_bonnet}.
\end{proof}

\begin{remark}\label{rem:GB_interpretation}
Equation~\eqref{eq:GB_decomp} expresses the total Forman curvature of the prism complex as the sum of snapshot Euler characteristics corrected by the persistent-overlap structure of the contact sequence.
In particular, the contribution of prism subdivisions to $\mathfrak{F}(\K_{\mathrm{ST}})$ depends only on the persistent simplices rather than on the full structure of $F_t$.
This identity provides a network-level invariant of $\mathcal{C}$ that remains parameter-free except for the selection of the slice gap $K$, and offers a method for comparing contact sequences without reference to individual edge curvature.
\end{remark}

\section{Empirical analysis}\label{sec:analysis}

This section presents numerical experiments that quantify the behavior of the two Forman variants on the spatiotemporal prism complex.
Three classes of synthetic contact-sequence generators with increasing levels of physical realism, along with an empirical face-to-face contact dataset from the SocioPatterns Collaboration ~\cite{Cattuto2010}, were considered.
The default temporal-edge weight was considered as $g(\dt) = (1+\dt)^{-1}$, and diagonal edges were assigned, in the sense of Remark~\ref{rem:simplicial}, the half-attenuated weight as
$\half g(\dt)$.  
The weights of spatial edges and vertices were set to unity, whereas triangle weights were defined as the geometric mean of the weights of their constituent edges.
Coarse graining of the empirical data was achieved through simple binning with a selected window size $\Delta T$. All contacts within a single bin were treated as simultaneous, and the resulting binned sequence was denoted as $\mathcal{C}_{\Delta T}$.

\subsection{Contact streams}\label{ssec:streams}

For this study, three families of synthetic contact sequences and one empirical dataset were considered.

\paragraph{Synthetic models.}
Three families of contact sequences were generated on a fixed node set $V$ with $|V|=25$, each over a window of $T=50$ unit time steps.

\textit{Erd\H os--R\'enyi (ER) baseline.}
Each unordered pair $(i,j) \in V \times V$ generated contacts as independent Poisson processes at a rate of $\lambda=10^{-2}$ on $[0,T]$.

\textit{Activity-driven (AD) model \cite{Perra2012}.}
Each node $i$ was assigned an activation rate $a_i$ drawn from a truncated power-law distribution $p(a)\propto a^{-\alpha}$, defined on the interval $a\in[0.05,0.5]$ with exponent $\alpha=2.5$.
At every integer time step, node $i$ was activated with probability $a_i$ and formed $m=2$ links with randomly chosen partners.
This mechanism generated strongly heterogeneous node activity characteristic of activity-driven temporal networks.

\textit{Bursty model.}
The exponential inter-event distribution of the ER baseline was replaced by a Weibull distribution having shape parameter $k=0.5$, corresponding to a heavy-tailed regime, and with a scale parameter selected such that the expected total number of contacts equated to that of the ER baseline.
The resulting contact network exhibited bursty temporal dynamics \cite{Karsai2011}.

For all three synthetic networks, the temporal data were coarse-grained using a bin width $\Delta T = 5$ before constructing the prism complex.

\paragraph{Empirical data: SocioPatterns Hypertext 2009.}
The publicly available Hypertext 2009 (HT09) face-to-face contact dataset~\cite{Cattuto2010,Isella2011} recorded at the ACM Hypertext 2009 conference using wearable RFID badges with $20\,$s temporal resolution was used.
A representative one-hour window was extracted from the active part of the conference and bin contacts to $\Delta T=300\,$s windows, yielding a contact sequence of comparable size to the synthetic experiment.

\subsection{Model sizes and global statistics}\label{ssec:results_global}
Table~\ref{tab:results_summary} summarizes, for each contact network, the number of $1$-simplices and $2$-simplices in the resulting prism complex, the fraction of $1$-simplices for which the original and augmented Forman curvatures differ numerically, the mean curvatures $\overline{\F}$ and $\overline{\Faug}$, and the Pearson correlation coefficient between the two curvature measures over the set of $1$-simplices.
\begin{table}[H]
\caption{Global statistics of the prism complex (built with maximum slice gap $K=3$) and Forman--Ricci curvatures on three model networks and an empirical HT09 dataset.
``\%\,disagree'' counts $1$-simplices for which $|\F(e)-\Faug(e)|>10^{-9}$.}
\label{tab:results_summary}
\begin{ruledtabular}
\begin{tabular}{lrrrrrr}
Model & $|\K^{(1)}|$ & $|\K^{(2)}|$ & \%\,disagree
      & $\overline{\F}$ & $\overline{\Faug}$ & corr.\\
\hline
ER     &  440 &  31 & 67.3\,\% & $-7.33$  & $-4.12$ & 0.90\\
AD     &  562 &  46 & 56.0\,\% & $-3.93$  & $-2.10$ & 0.83\\
Bursty &  914 & 303 & 65.0\,\% & $-12.99$ & $-7.80$ & 0.89\\
HT09   & 1639 & 248 & 63.3\,\% & $-11.83$ & $-7.55$ & 0.93\\
\end{tabular}
\end{ruledtabular}
\end{table}

After calculating the two Forman--Ricci curvatures for the four models, the following three aspects are determined:
\begin{itemize}
    \item 
    The two curvature variants differ on a substantial fraction of the $1$-simplices, ranging from $56\%$ to $67\%$ across all four networks.
    This result differs from the agreement under uniform weights established in Lemma~\ref{lem:agreement}, and quantitatively confirms the prediction that interval-dependent temporal-edge weights generically break the equivalence.
    \item 
    Despite this point-wise disagreement, the two curvatures remain strongly Pearson-correlated ($\rho\geq 0.83$ in every case), indicating that they produce largely the same ordering of edges by curvature even when their numerical values differ. 
    \item The mean curvatures satisfy $\overline{\F}<\overline{\Faug}<0$ in all four cases: the augmented form systematically over-estimates the curvature relative to the original one.
\end{itemize}  

\subsection{Comparison of two types of the Forman curvature}

The two complementary views of the per-edge relationship between $\F$ and $\Faug$ are the edge-by-edge scatter plots (Fig. ~\ref{fig:scatter}) and population histograms (Fig. ~\ref{fig:hist}).

Figure~\ref{fig:scatter} presents the edge-by-edge scatter plot of $\Faug$ against $\F$ for all four contact networks, color-coded using the three edge classes from Remark~\ref{rem:simplicial}.
Spatial edges (intra-snapshot contacts; same-time endpoints) lie almost exactly on the line $y=x$, in agreement with the analysis, as their two endpoint weights are equal (both unity by default), and their incident temporal edges enter symmetrically, such that the discrepancies in \eqref{eq:forman_orig} and \eqref{eq:forman_aug} cancel each other.
By contrast, the temporal and diagonal edges scatter increasingly far from $y=x$ as the time gap $\dt$ increases, exactly in the regime where the temporal weight $g(\dt)$ most strongly departs from unity.

\begin{figure}[H]
\centering
\includegraphics[width=\columnwidth]{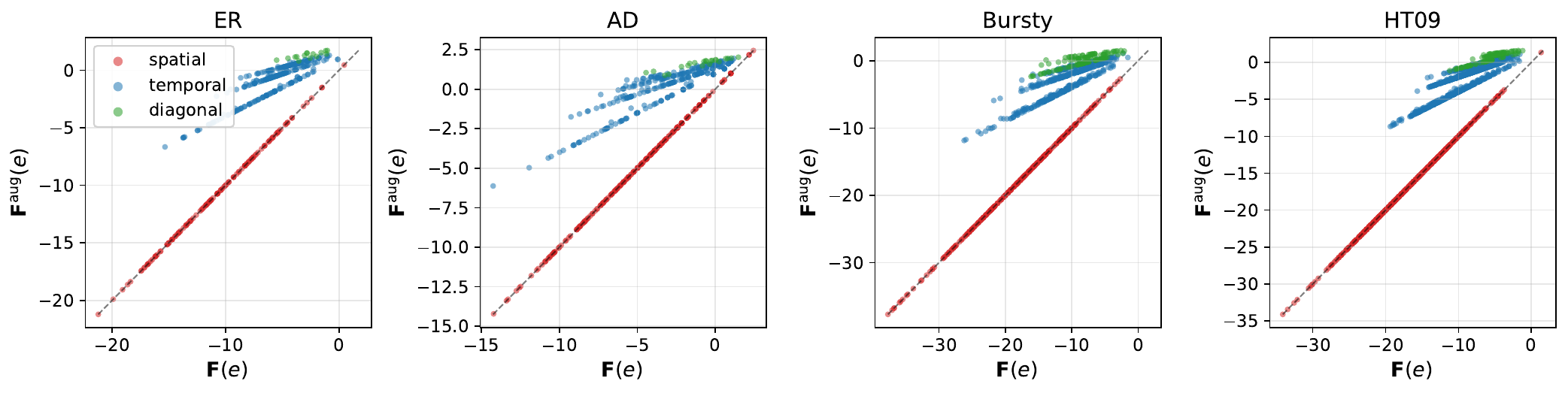}
\caption{Per-edge comparison of $\Faug$ against $\F$ for the four contact networks of Sec.~\ref{ssec:streams}.
Colors distinguish spatial (red), temporal (blue), and diagonal (green) edges of the prism complex.
Dashed line: identity $y=x$.}
\label{fig:scatter}
\end{figure}

Figure~\ref{fig:hist} presents the histograms of $\F$ and $\Faug$ for each network.
The result demonstrates that $\Faug$ is systematically less negative, which aligns with the global statistic $\overline{\F}<\overline{\Faug}<0$.
The bursty model and HT09 dataset generate the broadest negative tails, reflecting a larger number of bottleneck-like edges concentrated in the active periods of these networks.

A practical consequence of this systematic shift is that the original Forman curvature $\F$ exhibits a longer and more pronounced negative tail than $\Faug$ for each contact network considered.
As the edges with strongly negative Forman curvatures are the standard geometric proxies for bottleneck structures---edges that bridge otherwise weakly connected regions of the complex---the original variant assigns a wider dynamic range to these structurally significant edges.
However, $\F$ does not identify a different set of bottleneck edges than $\Faug$: as already noted, the rank-correlation between the two is high ($\rho\geq 0.83$ in Table~\ref{tab:results_summary}), such that both variants agree on which edges are the most bottleneck-like.
Rather, $\F$ resolves the magnitudes of these bottlenecks more sharply, making it the preferred descriptor whenever curvature values are to be compared across edges or aggregated into distribution-level statistics---the methodological distinction discussed in Section ~\ref{ssec:methodology}.

\begin{figure}[H]
\centering
\includegraphics[width=\columnwidth]{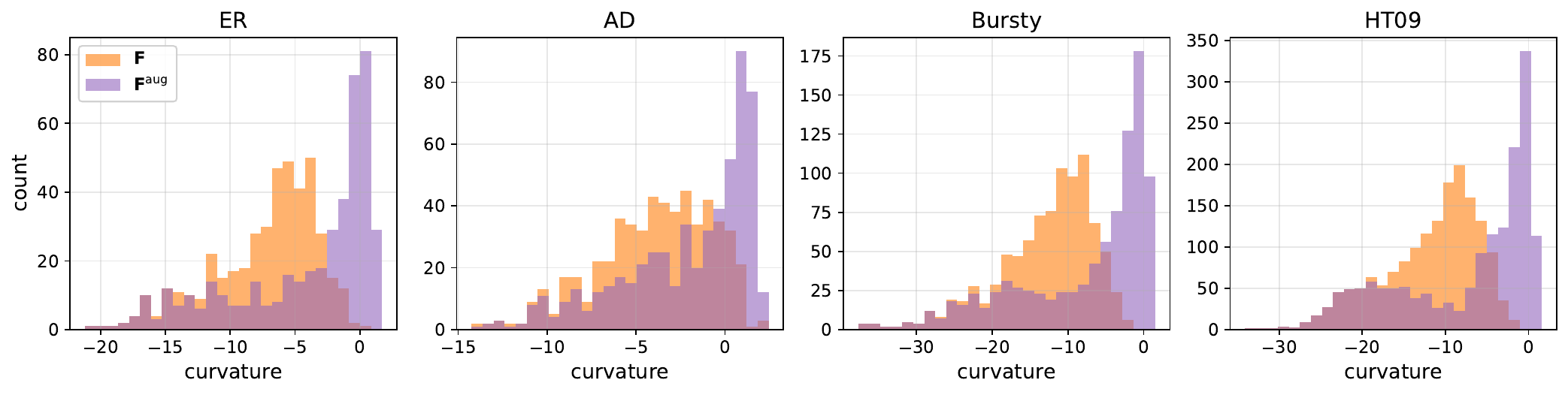}
\caption{Histograms of original (orange) and augmented (purple) Forman--Ricci curvatures on the four contact networks.}
\label{fig:hist}
\end{figure}

\subsection{Edge class statistics}\label{ssec:by_kind}
Figure~\ref{fig:by_kind} shows the mean original Forman curvature broken down by the edge class, with the error bars indicating the standard error of the mean.  Several physically interpretable patterns can be observed.

\begin{itemize}
    \item In all networks, spatial edges attain the most negative mean curvature, suggesting that simultaneous contacts are geometrically the most bottleneck-like elements of the prism complex.
    This result is consistent with the standard reading of Forman--Ricci curvature on static graphs: edges between hubs with disjoint neighborhoods are negatively curved.
    \item In the AD model, the spread of activity rates concentrates numerous contacts on a few highly active nodes, generating a smaller number of more strongly negative spatial-edge curvatures.
    \item In the bursty model and in HT09, temporal edges are less negative than those in the ER baseline.
    This observation indicates the burstiness of contacts: persistent edges (those that survive from one time slice to the next) are surrounded by several simultaneous contacts, raising the local triangle count and hence, the upward Forman contribution.
    \item 
    Diagonal edges follow an intermediate ordering: $|\overline{\F}_{\rm diag}|$ is uniformly smaller than $|\overline{\F}_{\rm spatial}|$ and $|\overline{\F}_{\rm temporal}|$ across all four networks, reflecting the fact that diagonal edges---by their construction in the prism subdivision---are flanked by smaller parallel-edge populations than spatial or temporal edges of comparable time gap.
    The relative size of $|\overline{\F}_{\rm diag}|$ between bursty/HT09 and ER/AD remains substantial: bursty/HT09 diagonal means lie at $\overline{\F}_{\rm diag}\approx -6$ to $-8$, whereas ER and AD give $\overline{\F}_{\rm diag}\approx -1$ to $-3$.
    The contrast scales with the local triangle density of the network around persistent simplices.
\end{itemize}

\begin{figure}[H]
\centering
\includegraphics[width=\columnwidth]{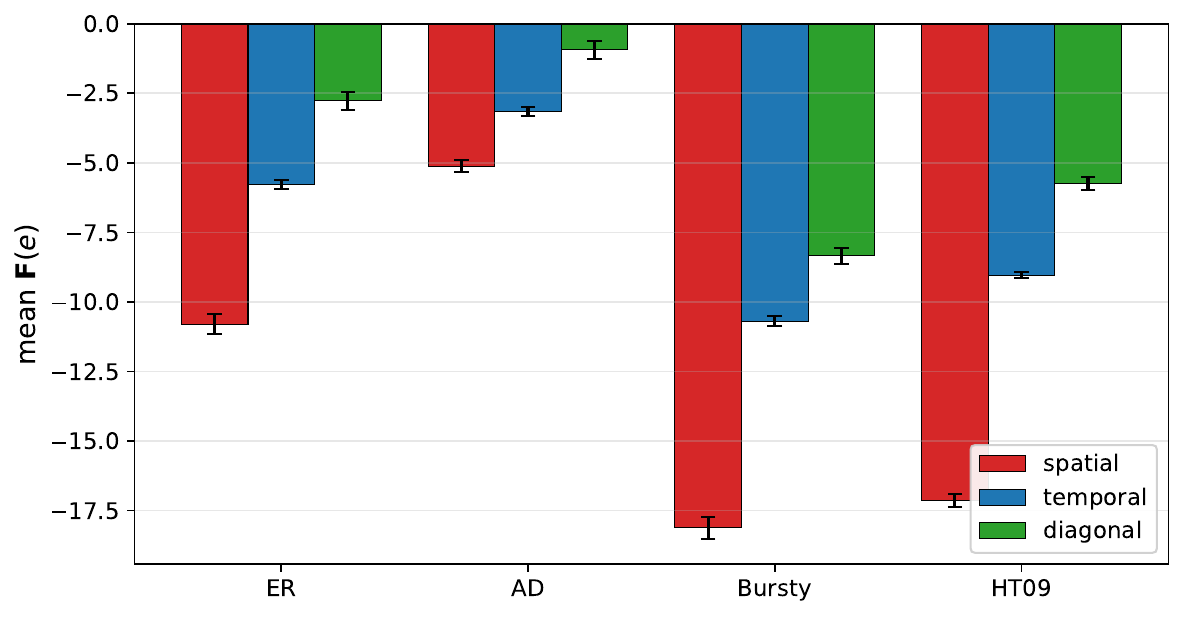}
\caption{Mean original Forman--Ricci curvature broken down by edge
class. Error bars indicate the standard error of the mean.}
\label{fig:by_kind}
\end{figure}

\subsection{Empirical validation of Corollary~\ref{cor:prism_specialization}}\label{ssec:dt_dep}

Although Corollary~\ref{cor:prism_specialization} provides the exact pointwise discrepancy $\F(e)-\Faug(e)$ for any temporal or diagonal edge $e$, Equation~\eqref{eq:temp_diag_disc} contains a factor $\sum_{\hat e\in\mathcal{P}(e)}\sqrt{w(e)/w(\hat e)}$ that depends on the local parallel-edge structure and is not a function of $\Delta t$ alone.
Accordingly, two questions arise.
First, does this network-dependent factor introduce significant scattering around the $\Delta t$-driven trends?
Second, do the model and real contact-stream complexes follow closed-form predictions, or does the implementation introduce numerical artifacts?
We address both questions by comparing the per-edge discrepancy with the theoretical predictions across the four networks.
With the maximum slice gap set to $K=3$, the temporal and diagonal edges of $\K_{\mathrm{ST}}$ cover the gaps $\Delta t \in \{1, 2, 3\}$ (and $\Delta t = 4$ for AD), allowing the $\Delta t$-dependence of $\F(e)-\Faug(e)$ for direct inspection.

Figure~\ref{fig:dt_dep} presents a feature that initially appears at odds with Equation ~\eqref{eq:temp_diag_disc}, where the magnitude $|\F(e)-\Faug(e)|$ offers a good approximation, independent of $\Delta t$ across the values $\Delta t \in \{1, 2, 3\}$ realized on the four networks.
This result is initially surprising considering the $(1-g(\Delta t))$-prefactor of Equation ~\eqref{eq:temp_diag_disc}, which grows monotonically with $\Delta t$ for any decreasing $g$.
The resolution lies in the network-dependent factor $\sum_{\hat e \in \mathcal{P}(e)} \sqrt{w(e)/w(\hat e)}$, which carries an implicit $\Delta t$-dependence through $w(e) = g(\Delta t)$.
Factoring out $\sqrt{g(\Delta t)}$ from this sum yields the equivalent expression
\begin{equation}\label{eq:hg_decomp}
|\F(e) - \Faug(e)|
\;=\;
\bigl(1 - g(\Delta t)\bigr)\sqrt{g(\Delta t)}
\sum_{\hat e \in \mathcal{P}(e)}\!\frac{1}{\sqrt{w(\hat e)}}.
\end{equation}
The first two factors combine into the function $h(g) = (1-g)\sqrt{g}$, which is non-monotonic on $(0,1)$ with a single maximum at $g = 1/3$ and is notably flat on the interval $g \in [1/4, 1/2]$ corresponding to $\Delta t \in \{1, 2, 3\}$ for $g(\Delta t) = 1/(1+\Delta t)$: $h$, takes the values $0.354, 0.385, 0.375$ at $\Delta t = 1, 2, 3$, varying by less than $9\%$.
The remaining factor $\sum_{\hat e} 1/\sqrt{w(\hat e)}$ depends on the local parallel-edge population but is empirically $\Delta t$-independent in the experiments (the average $|\mathcal{P}(e)|$ varies by less than $10\%$ across the three slice gaps, and the average parallel-edge weight by less than $5\%$).
The product of these two near-constant factors is almost constant, which explains the empirical observations in Figure ~\ref{fig:dt_dep}.
The vertical spread at a fixed $\Delta t$ reflects the variability of $\sum_{\hat e}1/\sqrt{w(\hat e)}$ as the local parallel-edge population changes from edge to edge.
The spread is broader on the bursty stream and HT09 dataset than on the ER, consistent with the higher density of $2$-simplices in these networks (Table~\ref{tab:results_summary}), producing more heterogeneous parallel-edge populations.

\begin{figure}[H]
\centering
\includegraphics[width=\columnwidth]{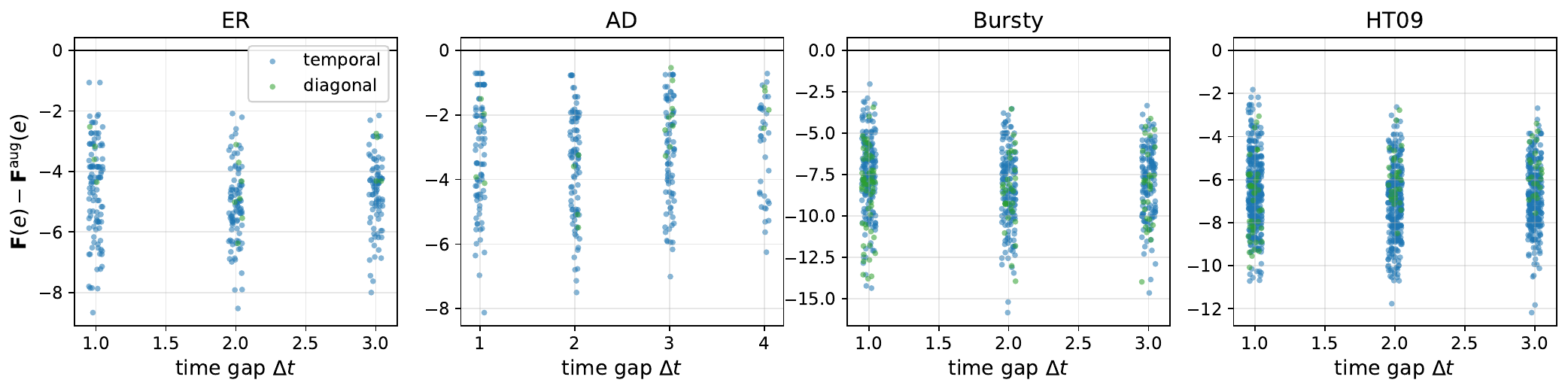}
\caption{Discrepancy $\F(e)-\Faug(e)$ as a function of the time gap $\dt$ for temporal (blue) and diagonal (green) edges.
At each fixed $\Delta t$, the discrepancy is concentrated below zero with a roughly symmetric vertical spread, and the typical magnitude is comparable across $\Delta t \in \{1, 2, 3\}$, consistent with the near-constancy of $h(g) = (1-g)\sqrt{g}$ on the corresponding range of $g$.}
\label{fig:dt_dep}
\end{figure}

The numerical results in this section reinforce the three points that are already visible in the theoretical analysis, provided in Section ~\ref{sec:theory_compare}.
First, on a generic contact-sequence-induced prism complex, the original and augmented Forman variants differ point-wise on a substantial fraction of edges, but they remain strongly correlated with the ordering statistics.
Second, the discrepancy is concentrated on temporally extended ($\dt>0$) edges, which are elements that distinguish temporal from static analysis.
Third, the qualitative features of the curvature distribution, such as the location and depth of the negative tail, are robust across the two variants, allowing one to use the augmented form as a fast proxy when only ordering information is needed, whereas the original form is required whenever the magnitude of the curvature must be compared across edges of different temporal extents.

\section{Discussion and outlook}\label{sec:conclusion}

This section introduces a spatiotemporal prism complex
$\K_{\mathrm{ST}}(\mathcal{C})$ as a rigorous geometric scaffold for
contact-sequence temporal networks and analyzes the original
and augmented Forman--Ricci curvatures on this scaffold both theoretically and numerically.  In this section, we summarize the contributions of this study, place the results in the broader literature, and outline future directions.

\subsection{Summary of contributions}\label{ssec:summary}

\textit{Geometric framework.}  The construction of
$\K_{\mathrm{ST}}(\mathcal{C})$ using the Hatcher prism operator
\cite{Hatcher2002} (Definition~\ref{def:prism_cx}) provides a
genuinely simplicial encoding of a contact sequence in which the
temporal ordering and interevent statistics are preserved.  Unlike
hybrid clique rule construction, prism complex automatically
generates diagonal $1$-simplices that mediate between
simultaneous spatial connectivity and temporal persistence.

\textit{Theoretical comparison of Forman variants.}  In a simplicial complex, the parallel XOR condition reduces to
``share a vertex but not a triangle'' (Lemma~\ref{lem:reduction}). The original CW-complex curvature \cite{Forman2003} and the augmented form \cite{Sreejith2016} coincide under uniform weights (Lemma~\ref{lem:agreement}), reducing to the elementary combinatorial expression in equation~\eqref{eq:uniform_value}.
Proposition~\ref{prop:discrepancy} provides a closed-form expression for the discrepancy $\F(e)-\Faug(e)$, expressed as Equation ~\eqref{eq:summand_factored}, revealing the mechanism of the disagreement.
For the prism complex with the default time-gap-dependent weight $g(\dt)$, Corollary~\ref{cor:prism_specialization} identifies which edges contribute to the discrepancy and how the contribution depends on $\dt$.

\textit{Structural properties of the prism complex.}
Beyond the comparison of curvature variants, the two theorems exploit the prism subdivision.
Theorem~\ref{prop:monotonicity} establishes that, under uniform weights along with controlled assumptions on parallel partners, the curvature of a persistent spatial simplex on $\K_{\mathrm{ST}}$ is bounded below by its static counterpart in $F_{t_k}$, with an excess equal to the difference between the number of prism-origin $2$-cofaces and parallel edges, offering a geometric explanation for the weakening of the negative temporal-edge means observed in the bursty and HT09 streams (Fig. ~\ref{fig:by_kind}).
Theorem~\ref{prop:gauss_bonnet} is a discrete Gauss--Bonnet identity for $\K_{\mathrm{ST}}$ that expresses the alternating Forman sum as the sum of snapshot Euler characteristics corrected by inclusion-exclusion terms over persistent simplices. This mechanism provides a parameter-free network-level invariant of the contact sequence that complements the per-edge curvature analysis.
A coupling lemma (Lemma~\ref{lem:coupling}) underlying the proof of Theorem~\ref{prop:monotonicity} decomposes the curvature of any spatial edge into a static intra-snapshot Forman--Ricci curvature and prism correction, whose support is identifiable in the contact data.

\textit{Empirical confirmation.}  Across three synthetic
contact network models (ER, activity-driven, bursty) and the
SocioPatterns Hypertext 2009 dataset, we deduce the following:
\textit{(i)} the two Forman variants disagree numerically on
$56$--$67\%$ of $1$-simplices, predominantly temporal and
diagonal edges while remaining strongly Pearson correlated.
($\rho\geq 0.83$), and \textit{(ii)} the spatial edges produce identical
curvatures under both definitions, which agree with
Propositions ~\ref{cor:prism_specialization} and \textit{(iii)}, the discrepancy
on the temporal/diagonal edges grows monotonically with the time gap
$\dt$, which quantitatively agrees with ~\eqref{eq:temp_diag_disc}.

\subsection{Methodological implications}\label{ssec:methodology}

The augmented Forman curvature \cite{Sreejith2016,Samal2018} has been
widely adopted in network science because its closed form involves
only the node degrees and triangle counts, lending itself to
$O(\!|\K^{(1)}|\!)$ computation on large graphs.  The obtained results
suggest a more nuanced picture of the temporal data.
\begin{itemize}
\item For rank-based analyses---identifying the most
      negatively curved edges, ranking edges as bottleneck candidates,
      thresholding for community separation---the high
      ($\rho\geq 0.83$) Pearson correlation justifies using the
      augmented form as a fast proxy.
\item For magnitude-based or cross-edge analyses---comparing curvatures of edges with different temporal extents,
      computing distribution moments, integrating curvature over
      time-windows of unequal length---the original form is
      necessary, as the augmented form systematically
      under-represents the negative tail by an amount that depends
      on the time-gap weighting $g(\dt)$.
\end{itemize}
Equation ~\eqref{eq:summand_factored} suggests that
practical diagnostics: A researcher can quickly assess whether the
augmented form is adequate for a given application by inspecting the
spread of $|w(v(\hat e))-w(e)|$ across the parallel edges.

\subsection{Position in the literature}\label{ssec:lit_position}
The existing geometric and topological approaches to temporal networks can be broadly classified into two categories.
The first is the \emph{snapshot }family.
This method computes the static curvatures on a sequence of binned graphs and tracks the resulting time series of the curvature distributions \cite{Bhaskar2021,Wee2023}.
The second type is a \emph{persistent family}.
This method directly applies a zigzag persistent homology to flag-complex filtration over time, producing barcoded diagnostics \cite{Myers2023}.
The prism complex lies between the two types: similar to the snapshot family, the prism complex provides edge-level local geometric quantities, and similar to the persistent family, it preserves the time ordering and explicitly integrates the temporal direction into the simplicial structure via the diagonal edges.
In particular, an aggregated curvature statistic such as $\sum_{e\in\K_{\mathrm{ST}}^{(1)}}\F(e)$ serves as a single scalar quantifier of the geometric structure that corresponds to the entire contact sequence and is parameter-free up to the selection of the weight function $g$ and optional binning width $\Delta T$.

\subsection{Future directions}\label{ssec:future}
\textit{Higher-dimensional Forman curvature.}  Forman's original
definition \eqref{eq:forman_orig} applies to all dimensions $p$, and
prism subdivision automatically produces $2$ and $3$-simplices
on $\K_{\mathrm{ST}}$.  Extending the analysis of
Section~\ref{sec:theory_compare} to $p\geq 2$ serves asa natural step,
requiring a generalization of
Proposition~\ref{prop:discrepancy} that accommodates the nontrivial
parallelism among $2$-simplices.

\textit{Optimal-transport curvatures.}  The Ollivier--Ricci
curvature \cite{Ollivier2009} provides an alternative transport-based
discretization, which is sensitive to different geometric features
than in Forman's combinatorial form.  Defining the Ollivier-style
curvature on $\K_{\mathrm{ST}}$ requires specifying a measure of
each spacetime vertex neighborhood that follows time ordering---a
nontrivial design choice that can be considered for future research.  A temporally biased random walker (one that prefers forward time-respecting neighbors) yields an Ollivier curvature on
$\K_{\mathrm{ST}}$, whose negative-tailed structure aligns with that of
$\F$.

\textit{Curvature flow on temporal networks.}  Discrete curvature
flows---iteratively reweighting the edges to drive the curvature
toward zero \cite{Weber2017}---have been used in Community
Detection and Surgery of Static Networks.  In a prism complex, 
analogous flows reweigh contacts as a function of their
local geometric positions, thus providing a principal smoothing operator
that respects time ordering.

\textit{Connections to dynamical processes.}  The temporal network
analogues of the spectral gap and mixing time, which control
spreadability and consensus convergence depend
sensitively on the time order and burstiness of contacts
\cite{Karsai2011,Scholtes2014}.  Prism-complex Forman curvature
is a candidate local geometric proxy for the spectral
quantities of the static Lott--Sturm--Villani-type
inequalities.  Establishing such an inequality on
$\K_{\mathrm{ST}}$---relating curvature lower bounds to
spreading-process upper bounds---links the proposed framework directly
to the rich body of literature on the dynamics of temporal networks.

\textit{Empirical applications.}  Beyond Hypertext 2009, the
SocioPatterns project \cite{Cattuto2010} provides face-to-face
contact data from primary schools, hospitals, and workplaces.
Comparison of curvature distributions across these settings---controlling for binning width, network size, and observation
horizon---test whether geometric signatures distinguish
context, as identified by other temporal network observables.
Beyond face-to-face data, financial transaction networks, neural network spike trains, and online communication logs are natural targets of
the same machinery.

In summary, the spatiotemporal prism complex offers a
mathematical principle and computationally tractable framework
for the geometric analysis of temporal networks. Forman's
original CW-complex Ricci curvature, used in conjunction with this
framework, serves as a reliable indicator of the bottleneck-like temporal
structure, in which the popular augmented form cannot be fully reproduced.

\section*{Acknowledgment}
The author would like to thank all the reviewers who helped improve this manuscript. 

\section*{Funding}
This study was supported by JSPS KAKENHI [grant number: 25K17253, 25K21661].
\section*{Conflict of interest}
None.
\section*{Declaration of generative AI and AI-assisted technologies in the manuscript preparation process.}
During the preparation of this work, the author has used Claude assistance to organize the manuscript structure and summarize related work. After using this tool, the author has reviewed and edited the content as required. The author takes full responsibility for the content of the published article.

\appendix

\section{Detailed proof of Theorem~\ref{prop:gauss_bonnet}}\label{app:gauss_bonnet}

This appendix provides a rigorous and self-contained proof of the
the Gaussian--Bonnet identity stated in Theorem~\ref{prop:gauss_bonnet}.
The proof is purely combinatorial and avoids appeals to the homological
invariants.
This notation has been retained in Section ~\ref{sec:theory_compare}: time slices
$T(\mathcal{C}) = \{t_1, \ldots, t_M\}$ are listed in the increasing order,
and $\mathcal{T}_K = \{(t_i, t_j) : 1 \le i < j \le M,\ j - i \le K\}$
represents the set of slice pairs from Definition~\ref{def:prism_cx};

\subsection{Preparatory lemmas}\label{app:prep}

Three lemmas are used in the main argument.

\begin{lemma}[Forman's combinatorial Gauss--Bonnet]\label{lem:forman_GB}
For any finite-weighted CW complex $K$ with positive weights,
\begin{equation*}\label{eq:forman_GB}
    \sum_{p \ge 0} (-1)^p \!\sum_{\alpha \in K^{(p)}} \F(\alpha)
    \;=\; \chi(K),
\end{equation*}
where $\chi(K) = \sum_p (-1)^p |K^{(p)}|$ is the combinatorial Euler
characteristic.
\end{lemma}
This expression is Forman's Theorem 4.1~\cite{Forman2003}. The theorem is used without further proof.

\begin{lemma}[Inclusion--exclusion for the Euler characteristic]
\label{lem:incl_excl}
Let $K$ be a finite simplicial complex covered by subcomplexes
$X_1, \ldots, X_N$ because $K = \bigcup_{a=1}^N X_a$ and
each pairwise intersection $X_a \cap X_b$ is a subcomplex of $K$.
Then
\begin{equation*}\label{eq:incl_excl}
    \chi(K)
    \;=\;
    \sum_{\emptyset \neq A \subseteq \{1,\ldots,N\}}
        (-1)^{|A|+1}\,
        \chi\Bigl(\bigcap_{a \in A} X_a\Bigr).
\end{equation*}
\end{lemma}
This is followed by the induction of $N$ in the case $N = 2$:
$\chi(X_1 \cup X_2) = \chi(X_1) + \chi(X_2) - \chi(X_1 \cap X_2)$,
which is the additivity of the alternating simplex count
$\sum_p (-1)^p |\,\cdot\,^{(p)}|$.

\begin{lemma}\label{lem:prism_euler}
For any nonempty simplex $\sigma \subseteq V$ and any pair of times
$t < t'$, the prism $\mathrm{Pr}(\sigma; t, t')$ from
Definition~\ref{def:prism_op} satisfies the following condition:
$\chi(\mathrm{Pr}(\sigma; t, t')) = 1$,
\end{lemma}

\begin{proof}
Considering $\sigma = \{v_0, v_1, \ldots, v_n\}$ with the chosen vertex
ordering, we have $\hat v_i = (v_i, t)$ and $w_i = (v_i, t')$.
By inspecting the top simplices in Equation~\eqref{eq:prism}, we have
\[
    S_i = \{\hat v_0, \ldots, \hat v_i, w_i, \ldots, w_n\}
    \qquad (i = 0, \ldots, n),
\]
Every top simplex of $\mathrm{Pr}(\sigma; t, t')$ contains $\hat v_0$, and every simplex of $\mathrm{Pr}(\sigma; t, t')$ that lies above $\hat v_0$ in face order saves the empty simplex.
Conversely, $\mathrm{Pr}(\sigma; t, t')$ is a simplicial cone with apex $\hat v_0$ over link $L := \{\tau \in \mathrm{Pr}(\sigma; t, t') : \hat v_0 \notin \tau\}$.

For any finite simplicial complex $L$, cone $\mathrm{Cone}_{x}(L) =
\{x\} \cup L \cup \{\{x\} \cup \tau : \tau \in L\}$ has Euler
characteristic $1$ because 
\begin{eqnarray*}
    \chi(\mathrm{Cone}_x(L))
    &=& 1 + \chi(L) + \sum_{\tau \in L}(-1)^{\dim\tau + 1}\\
    &=& 1 + \chi(L) - \chi(L) = 1.
\end{eqnarray*}
Applying this to $L$, we obtain
$\chi(\mathrm{Pr}(\sigma; t, t')) = 1$,
\end{proof}

\begin{lemma}
\label{lem:prism_snapshot}
Let $\sigma \in F_t \cap F_{t'}$ and $t_l$ denote any active time.
Subsequently, in the simplicial complex $\K_{\mathrm{ST}}(\mathcal{C})$, we have
\begin{equation*}\label{eq:prism_meet_snap}
    \mathrm{Pr}(\sigma; t, t') \cap F_{t_l}
    \;=\;
    \begin{cases}
        \sigma_t & \text{if } t_l = t,\\
        \sigma_{t'} & \text{if } t_l = t',\\
        \emptyset & \text{otherwise,}
    \end{cases}
\end{equation*}
where $\sigma_{t_l} := \{\{(v, t_l) : v \in \tau\} : \tau \subseteq \sigma\}$
represents the embedding of face complex of $\sigma$ at time $t_l$.
Specifically, $\chi(\mathrm{Pr}(\sigma; t, t') \cap F_{t_l}) = 1$ when
$t_l \in \{t, t'\}$, and $= 0$ otherwise.
\end{lemma}

\begin{proof}
The vertex set of $\mathrm{Pr}(\sigma; t, t')$ is $\{(v, t): v \in \sigma\}
\cup \{(v, t') : v \in \sigma\}$.
The vertex set of $F_{t_l}$ is $\{(v, t_l) : v \in V_{t_l}\}$.
A simplex $\tau$ exists in both complexes if every vertex of $\tau$ has a
time coordinates are equal to both $t_l$ and (one of) $\{t, t'\}$, which is
possible only when $t_l \in \{t, t'\}$.
The case $t_l = t$ yields $\tau \subseteq \{(v, t) : v \in \sigma\}$,
where $\tau$ is the face of $\sigma_t$; conversely, every face of $\sigma_t$
lie in both $\mathrm{Pr}(\sigma; t, t')$ and $F_t$.
The case $t_l = t'$ is symmetric, and $t_l \notin \{t, t'\}$ yields
the empty intersection.
The Euler characteristic of $\sigma_{t_l}$ (a single face complex of simplex) is $1$, and that of the empty complex is $0$.
\end{proof}

\subsection{Proof of Theorem~\ref{prop:gauss_bonnet}}\label{app:GB_proof}

\paragraph{First equality.}
$\K_{\mathrm{ST}}(\mathcal{C})$ is a finite simplicial complex, with
all simplex weights equal to $1$, hence, the finite
weighted CW complexes with a positive weight.
Lemma~\ref{lem:forman_GB} applied to $\K_{\mathrm{ST}}$ yields
$\mathfrak{F}(\K_{\mathrm{ST}}) = \chi(\K_{\mathrm{ST}})$
equality of equation~\eqref{eq:GB_decomp}.

\paragraph{Second equality.}
$\chi(\K_{\mathrm{ST}})$ is expanded using the cover from
Definition~\ref{def:prism_cx}.
For each $(i, j) \in I_K := \{(i, j) : 1 \le i < j \le M,\ j - i \le K\}$
let
\begin{equation*}
    \mathrm{Pr}_{i,j}
    \;:=\;
    \bigcup_{\sigma \in F_{t_i} \cap F_{t_j}}
        \mathrm{Pr}(\sigma; t_i, t_j),
\end{equation*}
the union of all prisms between time slices $t_i$ and $t_j$.
Then
\begin{equation*}\label{eq:cover}
    \K_{\mathrm{ST}}
    \;=\;
    \bigcup_{i=1}^{M} F_{t_i}
    \;\cup\;
    \bigcup_{(i,j) \in I_K} \mathrm{Pr}_{i,j}.
\end{equation*}
These $M + |I_K|$ subcomplexes are labeled uniformly as $X_1, \ldots, X_N$
where $X_i = F_{t_i}$ for $i \le M$ and $X_{M + r}$ ranging over
$\{\mathrm{Pr}_{i,j} : (i,j) \in I_K\}$ for fixed enumeration.
Lemma~\ref{lem:incl_excl} applied to this cover yields
\begin{equation}\label{eq:GB_full_expand}
    \chi(\K_{\mathrm{ST}})
    \;=\;
    \sum_{\emptyset \neq A \subseteq \{1,\ldots,N\}}
        (-1)^{|A|+1}\,
        \chi\Bigl(\bigcap_{a \in A} X_a\Bigr).
\end{equation}

The terms in Equation ~\eqref{eq:GB_full_expand} can be classified according to the type of
$A$.

\textbf{Type (S):} $|A| = 1$ and $A \subseteq \{1, \ldots, M\}$.
The contribution is $\sum_{i=1}^{M} \chi(F_{t_i})$.

\textbf{Type (P):} $|A| = 1$ and $A \subseteq \{M+1, \ldots, N\}$.
From Lemma~\ref{lem:incl_excl}, we have
\begin{eqnarray*}
    \chi(\mathrm{Pr}_{i,j}) &=& \chi \Bigl(\bigcup_{\sigma \in F_{t_i} \cap F_{t_j}}\!\mathrm{Pr}(\sigma; t_i, t_j)\Bigr)\\
    &=& \!\!\sum_{\emptyset \neq B \subseteq F_{t_i} \cap F_{t_j}}\!\!(-1)^{|B|+1}\,\chi\Bigl(\!\bigcap_{\sigma \in B}\!\mathrm{Pr}\bigl(\sigma;\,t_i, t_j\bigr)\Bigr).
\end{eqnarray*}
For any distinct face $\sigma_1, \sigma_2 \in F_{t_i} \cap F_{t_j}$, we have
\begin{equation*}
    \mathrm{Pr}(\sigma_1; t_i, t_j) \cap \mathrm{Pr}(\sigma_2; t_i, t_j) =\mathrm{Pr}(\sigma_1 \cap \sigma_2; t_i, t_j),
\end{equation*}
and from Lemma~\ref{lem:prism_euler}, we obtain
\begin{eqnarray*}
    \chi(\mathrm{Pr}_{i,j}) &=&  \!\!\sum_{\emptyset \neq B \subseteq F_{t_i} \cap F_{t_j}}\!\!(-1)^{|B|+1}\,\cdot 1= \chi(F_{t_i}\cap F_{t_j}).
\end{eqnarray*}
Thus, the contribution is 
\begin{equation*}
    \text{Type (P) contribution} = \sum_{(i,j) \in I_K} \chi(F_{t_i}\cap F_{t_j}).
\end{equation*}

\textbf{Type (SS):} $|A| = 2$ and $A \subseteq \{1, \ldots, M\}$.
For each unordered pair $(i_1, i_2)$ with $i_1 < i_2$, the
contribution is $-\chi(F_{t_{i_1}} \cap F_{t_{i_2}})$.

\textbf{Type (SP):} $|A| = 2$ with one snapshot $X_i = F_{t_i}$ and
a one-prism stack $X_{M+r} = \mathrm{Pr}_{i_a, j_a}$.
From Lemma~\ref{lem:prism_snapshot}, we have
\[
F_{t_i} \cap \mathrm{Pr}_{i_a, j_a}
= \begin{cases}
    \bigcup_{\sigma \in F_{t_{i_a}} \cap F_{t_{j_a}}} \sigma_{t_i}
    & \text{if } i \in \{i_a, j_a\},\\
    \emptyset & \text{otherwise.}
\end{cases}
\]
When $i = i_a$, the intersection is an embedding of
$F_{t_{i_a}} \cap F_{t_{j_a}}$ at time $t_{i_a}$ in $F_{t_{i_a}}$;
its Euler characteristic is $-\chi(F_{t_{i_a}} \cap F_{t_{j_a}})$.
The case $i = j_a$ is symmetric.
For each $(i_a, j_a) \in I_K$, exactly two such nonvanishing
SP contributions (at $i = i_a$ and $i = j_a$) exist, with each contributing
$-\chi(F_{t_{i_a}} \cap F_{t_{j_a}})$.
Therefore
\begin{equation*}
    \text{Type (SP) contribution}
    = -2 \sum_{(i,j) \in I_K} \chi(F_{t_i} \cap F_{t_j}).
\end{equation*}

\textbf{Type (PP):} $|A| = 2$ and $A \subseteq \{M+1, \ldots, N\}$.
We claim that such an intersection is empty unless the two prism
stacks share a common time slice.
In fact, $\mathrm{Pr}_{i_1, j_1} \cap \mathrm{Pr}_{i_2, j_2}$ contains
vertex $(v, t_l)$ only if $t_l \in \{t_{i_1}, t_{j_1}\}$ and
$t_l \in \{t_{i_2}, t_{j_2}\}$, hence
$\{t_{i_1}, t_{j_1}\} \cap \{t_{i_2}, t_{j_2}\} \neq \emptyset$.
When the two prism stacks share exactly one time slice $t_l$, the
intersection lies inside $F_{t_l}$ (by Lemma~\ref{lem:prism_snapshot}),
and it consists of the simultaneous embedding of
$F_{t_{i_1}} \cap F_{t_{j_1}}$ and $F_{t_{i_2}} \cap F_{t_{j_2}}$
at time $t_l$.

All such PP contributions are collected in the residual term
$\mathcal{R}(\mathcal{C})$.

\textbf{Higher-order Type (H):} $|A| \ge 3$, mixing snapshots and
prism stacks freely.
These terms are included in $\mathcal{R}'(\mathcal{C})$.

\paragraph{Combining all contributions.}
Combining Types~(S), (P), (SS), and (SP), the total Euler characteristic is given by
\begin{align*}
    \chi(\K_{\mathrm{ST}})
    &= \sum_{i=1}^M \chi(F_{t_i})
       \;+\;\!\!\sum_{(i,j) \in I_K}\!\! \chi(F_{t_i} \cap F_{t_j}) \\
    &\quad - \sum_{i_1 < i_2}\!\! \chi(F_{t_{i_1}} \cap F_{t_{i_2}})
       \;-\; 2\!\!\sum_{(i,j) \in I_K}\!\! \chi(F_{t_i} \cap F_{t_j}) \\
    &\quad + \mathcal{R}'(\mathcal{C}),
\end{align*}
where $\mathcal{R}'(\mathcal{C})$ collects all (PP) and higher-order $|A| \ge 3$ contributions.
We decompose the snapshot--snapshot sum (Type (SS)) into pairs inside $\mathcal{T}_K$ and pairs outside:
\begin{align*}
    \text{(SS)}
    &= - \sum_{i_1 < i_2} \chi(F_{t_{i_1}} \cap F_{t_{i_2}}) \nonumber\\
    &= -\!\!\sum_{(i,j) \in I_K}\!\! \chi(F_{t_i} \cap F_{t_j})
       \;-\!\!\sum_{\substack{i_1 < i_2 \\ (i_1, i_2) \notin I_K}}\!\!
            \chi(F_{t_{i_1}} \cap F_{t_{i_2}}).
\end{align*}
Thus, the (P) and (SP) terms, along with the $I_K$-portion of the (SS), sum to $-2\!\!\sum_{(i,j) \in I_K}\!\!\chi(F_{t_i} \cap F_{t_j})$. By adding the non-$I_K$ portion of (SS) to $\mathcal{R}'(\mathcal{C})$, we define 
\begin{equation*}
    \mathcal{R}(\mathcal{C}) := \mathcal{R}'(\mathcal{C}) - \sum_{(i_1,i_2) \notin I_K} \chi(F_{t_{i_1}} \cap F_{t_{i_2}}).
\end{equation*}
Then, by renaming the snapshot index to $t$ rather than $t_i$, and the slice pair to $(t, t') \in \mathcal{T}_K$ rather than $(i, j) \in I_K$, we obtain
\begin{equation*}
    \chi(\K_{\mathrm{ST}})
    \;=\;
    \sum_{t \in T(\mathcal{C})}\!\chi(F_t)
    \;-\;2\!\!\sum_{(t,t') \in \mathcal{T}_K}\!\!\chi(F_t \cap F_{t'})
    \;+\; \mathcal{R}(\mathcal{C}),
\end{equation*}
which is the second equality of Equation~\eqref{eq:GB_decomp}.
This completes the proof of Theorem~\ref{prop:gauss_bonnet}.

\qed

\subsection{Remarks on the residual}\label{app:residual}

The residual $\mathcal{R}(\mathcal{C})$ admits the explicit form
\begin{align}\label{eq:R_explicit}
    \mathcal{R}(\mathcal{C})
    &= -\!\!\sum_{\substack{(t, t') \in T(\mathcal{C})^2 \\ t < t', (t, t') \notin \mathcal{T}_K}}\!\!
            \chi(F_t \cap F_{t'}) \nonumber \\
    &\quad +\!\!\sum_{\substack{(i,j),(i',j') \in I_K \\ \{t_i, t_j\} \cap \{t_{i'}, t_{j'}\} \neq \emptyset \\ (i,j) \neq (i', j')}}\!\!
        (-1)^3\,\chi\bigl(\mathrm{Pr}_{i,j} \cap \mathrm{Pr}_{i',j'}\bigr) \nonumber \\
    &\quad + (\text{higher-order } |A| \ge 3 \text{ terms}).
\end{align}
The first sum vanishes when the time lapse between two active slices outside $\mathcal{T}_K$ contains no persistent simplex ($F_t \cap F_{t'}$ is empty for every non-$\mathcal{T}_K$ pair).
The second and third sums vanish when no simplex persists across three or more time slices (such that all triple intersections of prism stacks are reduced to the lower-order terms already accounted for).
In typical bursty contact streams, where the persistence is short-lived, $\mathcal{R}(\mathcal{C})$ is dominated by terms with small absolute values relative to the leading sums.
This explains why the practical computation of $\mathfrak{F}(\K_{\mathrm{ST}})$ via Equation~\eqref{eq:GB_decomp} can be approximated by the snapshot--minus--twice-slice-pair sum, with $\mathcal{R}(\mathcal{C})$ providing a controllable correction.

\bibliographystyle{plain}
\bibliography{references}

\end{document}